\documentclass{article}
\usepackage[final]{graphicx}
\usepackage{psfrag}
\usepackage{amsmath,amsfonts,amssymb,amsxtra,subeqnarray}

\newlength{\extramargin}
\newcommand {\Real}{\ensuremath{{\mathbb{R}}}}

\newcommand {\Complex}{\ensuremath{{\mathbb{C}}}}

\newcommand{\C}{\ensuremath{\mathcal C}}

\newcommand{\V}{\ensuremath{\mathcal V}}

\newcommand{\setS}{\ensuremath{\mathcal S}}

\newcommand{\U}{\ensuremath{\mathcal U}}

\newcommand{\N}{\ensuremath{\mathcal N}}

\newtheorem{theorem}{Theorem}

\newtheorem{corollary}{Corollary}

\newtheorem{lemma}{Lemma}

\newtheorem{definition}{Definition}

\newtheorem{proposition}{Proposition}
\newenvironment{proof}{\noindent {\bf Proof.}}{\hfill \hspace*{1pt}\hfill$\blacksquare$}

\begin{document}
\title{Observability through matrix-weighted graph}
\author{S. Emre Tuna\footnote{The author is with Department of
Electrical and Electronics Engineering, Middle East Technical
University, 06800 Ankara, Turkey. Email: {\tt
tuna@eee.metu.edu.tr}}} \maketitle

\begin{abstract}
Observability of an array of identical LTI systems with
incommensurable output matrices is studied, where an array is
called observable when identically zero relative outputs imply
synchronized solutions for the individual systems. It is shown
that the observability of an array is equivalent to the
connectivity of its interconnection graph, whose edges are
assigned matrix weights. The interconnection graph is studied by
means of a collection of simpler graphs, each of which is
associated to an eigenvalue of the system matrix of individual
dynamics. It is reported that the interconnection graph is
connected if and only if no member of this collection is
disconnected. Moreover, to better understand the relative behavior
of distant units, pairwise observability which concerns with the
synchronization of a certain pair of individual systems in the
array is studied. This milder version of observability is shown to
be closely related to certain connectivity properties of the
interconnection graph as well. Pairwise observability is also
analyzed using the circuit theoretic tool effective conductance.
The observability of a certain pair of units is proved to be
equivalent to the nonsingularity of the (matrix-valued) effective
conductance between the associated pair of nodes of a resistive
network (with matrix-valued parameters) whose node admittance
matrix is the Laplacian of the array's interconnection graph.
\end{abstract}

\section{Introduction}

Observability is one of the central concepts in systems theory
which, for LTI systems, can be expressed in many seemingly
different yet mathematically equivalent forms \cite{hespanha09}.
One alternative is the following. {\em A pair $[C,\,A]$ is
observable if $C(x_{1}(t)-x_{2}(t))\equiv 0$ implies
$x_{1}(t)\equiv x_{2}(t)$}, where $x_{i}(t)$ denote the solutions
of two identical systems ${\dot x}_{i}=Ax_{i}$, $i=1,\,2$.
Admittedly, this appears to be an uneconomical definition, for the
implication therein employs two systems where one would have
sufficed. The overuse, however, has a relative advantage: it
points in an interesting direction of generalization. Namely, for
a pair $[(C_{ij})_{i,j=1}^{q},\,A]$ the below condition suggests
itself as a natural extension.
\begin{eqnarray}\label{eqn:cond}
C_{ij}(x_{i}(t)-x_{j}(t))\equiv 0\ \mbox{for all}\ (i,\,j)
\implies x_{i}(t)\equiv x_{j}(t)\ \mbox{for all}\ (i,\,j)
\end{eqnarray}
where $x_{i}(t)$ are the solutions of $q$ identical systems ${\dot
x}_{i}=Ax_{i}$, $i=1,\,2,\,\ldots,\,q$. Aside from its theoretical
allure, this particular choice of generalization is not without
practical motivation; the condition~\eqref{eqn:cond} happens to be
both necessary and sufficient for synchronization of certain
arrays of oscillators. We give two examples in the sequel.

{\em Coupled electrical oscillators.} Consider the individual
oscillator in Fig.~\ref{fig:LCcircuit}, where $p$ linear inductors
with inductances $m_{1},\,m_{2},\,\ldots,\,m_{p}>0$ are connected
by linear capacitors with capacitances
$k_{1},\,k_{2},\,\ldots,\,k_{p+1}>0$.
\begin{figure}[h]
\begin{center}
\includegraphics[scale=0.55]{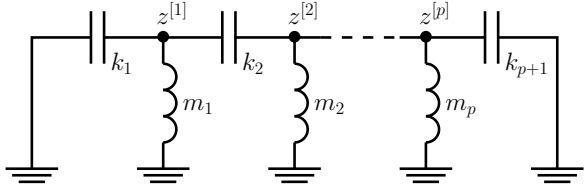}
\caption{Electrical oscillator.}\label{fig:LCcircuit}
\end{center}
\end{figure}
The node voltages are denoted by $z^{[\ell]}\in\Real$. Letting
$z=[z^{[1]}\ z^{[2]}\ \cdots\ z^{[p]}]^{T}$ the model of this
system reads $K{\ddot z}+M^{-1}z=0$ where $M={\rm
diag}(m_{1},\,m_{2},\,\ldots,\,m_{p})$ and
\begin{eqnarray*}
K=\left[\begin{array}{ccccc}k_{1}+k_{2}&-k_{2}&0&\cdots&0\\-k_{2}&k_{2}+k_{3}&-k_{3}&\cdots&0\\0&-k_{3}&k_{3}+k_{4}&\cdots&0\\
\vdots&\vdots&\vdots&\ddots&\vdots\\0&0&0&\cdots&k_{p}+k_{p+1}\end{array}\right]\,.
\end{eqnarray*}

Let now an array be constructed by coupling $q$ identical
oscillators in the arrangement shown in Fig.~\ref{fig:LCarray}. If
we let $z_{i}\in\Real^{p}$ denote the node voltage vector for the
$i$th oscillator and $b_{ij}^{[\ell]}=b_{ji}^{[\ell]}\geq 0$ be
the conductance of the resistor connecting the $\ell$th nodes of
the oscillators $i$ and $j$, we obtain $K{\ddot
z}_{i}+M^{-1}z_{i}+\sum_{j=1}^{q}B_{ij}({\dot z}_{i}-{\dot
z}_{j})=0$, where $B_{ij}={\rm
diag}(b_{ij}^{[1]},\,b_{ij}^{[2]},\,\ldots,\,b_{ij}^{[p]})$.
Denoting by $x_{i}=[z_{i}^{T}\ {\dot z}_{i}^{T}]^{T}$ the state of
the $i$th system we can then rewrite the coupled dynamics as
\begin{eqnarray}\label{eqn:LC-oscillator}
{\dot
x}_{i}=\left[\begin{array}{cc}0&I_{p}\\-K^{-1}M^{-1}&0\end{array}\right]x_{i}
+\sum_{j=1}^{q}\left[\begin{array}{cc}0&0\\0&K^{-1}B_{ij}\end{array}\right](x_{j}-x_{i})\,,\qquad
i=1,\,2,\,\ldots,\,q\,.
\end{eqnarray}

To understand the collective behavior of these coupled oscillators
one can employ the Lyapunov function
\begin{eqnarray}\label{eqn:lyap}
V(x_{1},\,x_{2},\,\ldots,\,x_{q})=\frac{1}{2}\sum_{i=1}^{q}x_{i}^{T}\left[\begin{array}{cc}M^{-1}&0\\0&K\end{array}\right]x_{i}\,.
\end{eqnarray}
In particular, combining \eqref{eqn:LC-oscillator} and
\eqref{eqn:lyap} we reach
\begin{eqnarray}\label{eqn:Vdot}
\frac{d}{dt}V(x_{1}(t),\,x_{2}(t),\,\ldots,\,x_{q}(t))=-\sum_{i>j}(x_{i}(t)-x_{j}(t))^{T}\left[\begin{array}{cc}0&0\\0&B_{ij}\end{array}\right](x_{i}(t)-x_{j}(t))\,.
\end{eqnarray}
Note that the righthand side is negative semidefinite because all
$B_{ij}$ are positive semidefinite. Hence the solutions
$x_{1}(t),\,x_{2}(t),\,\ldots,\,x_{q}(t)$ remain bounded. Now, for
\begin{eqnarray*}
A=\left[\begin{array}{cc}0&I_{p}\\-K^{-1}M^{-1}&0\end{array}\right]\quad
\mbox{and}\quad
C_{ij}=\left[\begin{array}{cc}0&0\\0&B_{ij}\end{array}\right]
\end{eqnarray*}
suppose that the condition~\eqref{eqn:cond} holds. Then, and only
then, \eqref{eqn:Vdot} yields by Krasovskii-LaSalle invariance
principle \cite{khalil00} that the oscillators synchronize, i.e.,
$\|x_{i}(t)-x_{j}(t)\|\to 0$ for all $(i,\,j)$ and all initial
conditions $x_{1}(0),\,x_{2}(0),\,\ldots,\,x_{q}(0)$.

\begin{figure}[h]
\begin{center}
\includegraphics[scale=0.55]{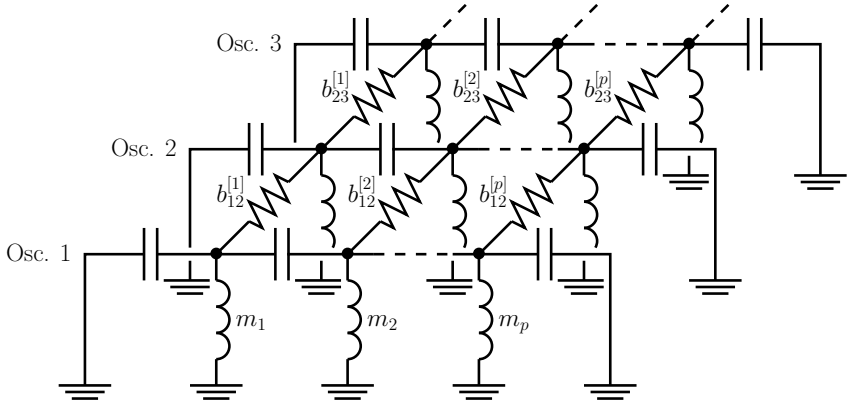}
\caption{Array of coupled electrical
oscillators.}\label{fig:LCarray}
\end{center}
\end{figure}

{\em Coupled mechanical oscillators.} Our second example employs
the mechanical system shown in Fig.~\ref{fig:masspring}, where $p$
masses are connected by linear springs. Such chains are used to
model the interaction of atoms in a crystal \cite{dove93}.
\begin{figure}[h]
\begin{center}
\includegraphics[scale=0.55]{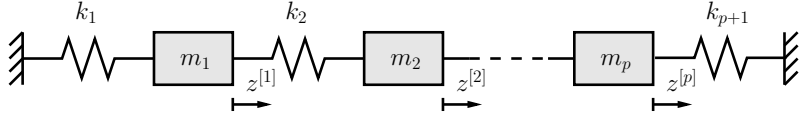}
\caption{Mechanical oscillator.}\label{fig:masspring}
\end{center}
\end{figure}
Let $z^{[\ell]}\in\Real$ be the displacement of the mass
$m_{\ell}>0$ from the equilibrium. The spring constants are
denoted by $k_{1},\,k_{2},\,\ldots,\,k_{p+1}>0$. Letting
$z=[z^{[1]}\ z^{[2]}\ \cdots\ z^{[p]}]^{T}$ the model of this
oscillator reads $M{\ddot z}+Kz=0$, where the matrices $M,\,K$ are
borrowed from the previous example.

Let now an array be formed by coupling $q$ replicas of this
oscillator in the arrangement shown in Fig.~\ref{fig:msparray}. If
we let $z_{i}\in\Real^{p}$ denote the displacement vector for the
$i$th oscillator and $b_{ij}^{[\ell]}=b_{ji}^{[\ell]}\geq 0$
represent the viscous friction (damping) between the $\ell$th
masses of the oscillators $i$ and $j$, we can write $M{\ddot
z}_{i}+Kz_{i}+\sum_{j=1}^{q}B_{ij}({\dot z}_{i}-{\dot z}_{j})=0$
with $B_{ij}={\rm
diag}(b_{ij}^{[1]},\,b_{ij}^{[2]},\,\ldots,\,b_{ij}^{[p]})$.
\begin{figure}[h]
\begin{center}
\includegraphics[scale=0.55]{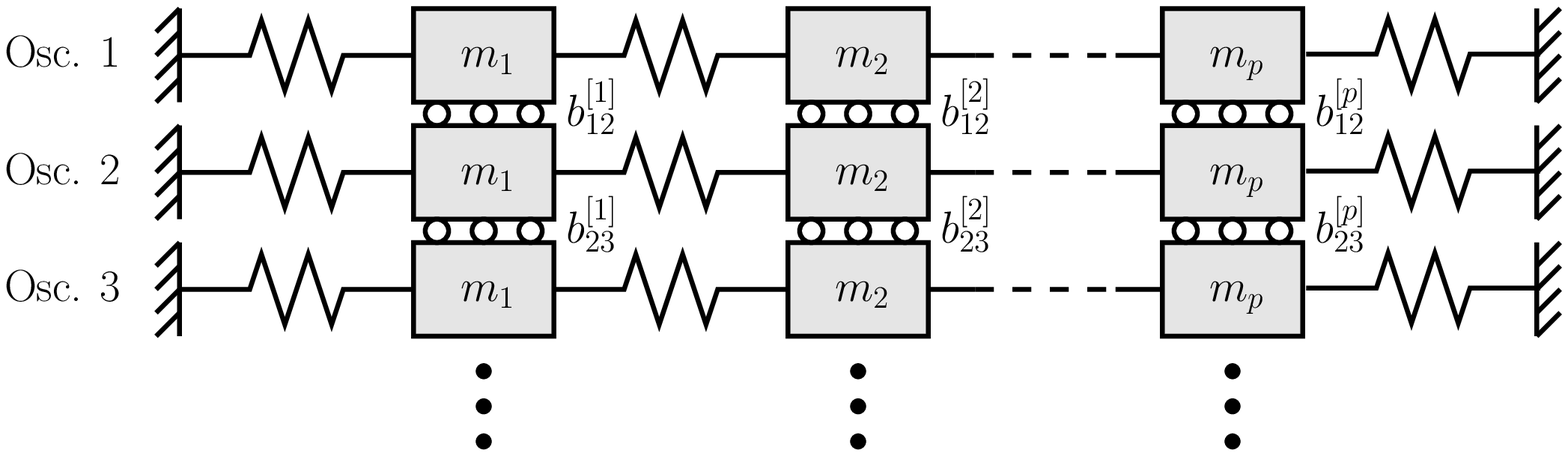}
\caption{Array of coupled mechanical
oscillators.}\label{fig:msparray}
\end{center}
\end{figure}
The Lyapunov approach previously adopted for the synchronization
analysis of the coupled electrical oscillators is valid here, too.
The outcome is the same. Namely, under the
condition~\eqref{eqn:cond}, this time with
\begin{eqnarray*}
A=\left[\begin{array}{cc}0&I_{p}\\-M^{-1}K&0\end{array}\right]\quad
\mbox{and}\quad
C_{ij}=\left[\begin{array}{cc}0&0\\0&B_{ij}\end{array}\right]\,,
\end{eqnarray*}
the mechanical oscillators synchronize. Having motivated the
condition~\eqref{eqn:cond} in the context of synchronization, we
will next try to explain its relation to certain existing
assumptions.

Synchronization of linear systems is a broad area of research,
where one of the main goals of the researcher is to unearth
conditions under which the solutions of coupled units converge to
a common trajectory. Different sets of assumptions have led to a
rich collection of results, bringing our understanding on the
subject closer to complete; see, for instance,
\cite{li10,zhang11,zhou14,scardovi09,grip12,li13}. Despite their
differences in degree and direction of generality, all these works
share two assumptions in common: (i) the graph describing the
interconnection contains a spanning tree and (ii) the individual
system is observable (detectable). We intend to emphasize in this
paper that these two separate assumptions, the former on {\em
connectivity} and the latter on {\em observability}, dissolve
inseparably in the condition~\eqref{eqn:cond}. In particular, for
an array represented by the pair $[(C_{ij})_{i,j=1}^{q},\,A]$, it
is in general not meaningful to search for a spanning tree because
the interconnection graph will be matrix-weighted, whereas a tree
is well defined for a scalar-weighted graph only. As for the
second assumption, requiring the individual systems to be
observable also falls prey to ambiguity since there is not a
single output matrix for each system; instead every system is
coupled to each of its neighbors through a different matrix
$C_{ij}$. It is true that separation is possible in the special
case $C_{ij}=w_{ij}C$ with $C\in\Complex^{m\times n}$ and
$w_{ij}\in\Real$. In this much-studied scenario, where the output
matrices are commensurable, the scalar weights $w_{ij}$ are used
to construct the interconnection graph, which can be checked to
contain a spanning tree; and the pair $[C,\,A]$ can separately be
checked for observability. However, in general, the
condition~\eqref{eqn:cond} in its entirety is what we have to deal
with, which requires that we work with the matrix-weighted graphs.
We will explain how these matrix-valued weights emerge soon. But
first, let us review the scarce literature on observability over
networks.

Observability over networks, motivated in general by
synchronization (consensus) of coupled systems, is largely an
unexplored area of research. Among the few works is \cite{ji07},
where the observability of sensor networks is studied by means of
equitable partitions of graphs. This tool is employed also in
\cite{oclery13}. The observability of path and cycle graphs is
studied in \cite{parlangeli12} and of grid graphs in
\cite{notarstefano13}. Recently, the networks whose individual
systems' dynamics are allowed to be nonidentical is covered in
\cite{zhou15}. Each of these investigations covers a different
case, yet they all consider interconnections that can be described
by graphs with scalar-weighted edges. At this point our work is
located relatively far from the reported results. In particular,
to the best of our knowledge, observability over matrix-weighted
graphs has not yet been studied in detail.

In the first half of this paper we report conditions on the array
$[(C_{ij})_{i,j=1}^{q},\,A]$ that imply observability in the sense
of \eqref{eqn:cond}. To this end, we construct a graph $\Gamma$
(with $q$ vertices) where to each pair of vertices
$(v_{i},\,v_{j})$ we assign a weight that is a Hermitian positive
semidefinite matrix, whose null space is the unobservable subspace
corresponding to the individual pair $[C_{ij},\,A]$. We reveal
that the array $[(C_{ij})_{i,j=1}^{q},\,A]$ is observable if and
only if the interconnection graph $\Gamma$ is connected. Also, we
notice that for each distinct eigenvalue of $A$ there exists a
graph (we call it an {\em eigengraph}) and the observability of
$[(C_{ij})_{i,j=1}^{q},\,A]$ is ensured if no eigengraph is
disconnected. For our analysis we define the connectivity of a
graph through a certain spectral property of its Laplacian. We
note that the connectivity of a matrix-weighted graph cannot in
general be characterized by the standard tools of graph theory
such as path and tree. The reason is that the meaning or function
of an edge (out of which paths and trees are constructed) becomes
equivocal when one has to allow semidefinite weights.

In the second half of the paper we focus on the so-called
$(k,\,\ell)$-observability of $[(C_{ij})_{i,j=1}^{q},\,A]$.
Namely, for a given pair of indices $(k,\,\ell)$, we search for
conditions under which $x_{k}(t)\equiv x_{\ell}(t)$ provided that
$C_{ij}(x_{i}(t)-x_{j}(t))\equiv 0\ \mbox{for all}\ (i,\,j)$. To
this end we define $(k,\,\ell)$-connectivity of a matrix-weighted
graph through its Laplacian. We show the expected equivalence
between the $(k,\,\ell)$-observability of the array
$[(C_{ij})_{i,j=1}^{q},\,A]$ and the $(k,\,\ell)$-connectivity of
the interconnection graph $\Gamma$ as well as the unexpected lack
of equivalence between the $(k,\,\ell)$-observability of the array
and the $(k,\,\ell)$-connectivity of its eigengraphs. Moreover, we
present the interesting interchangeability between the
$(k,\,\ell)$-observability of an array and the nonsingularity of
the matrix $\Gamma_{k\ell}$, where $\Gamma_{k\ell}$ is the
(matrix-valued) effective conductance between the nodes $k$ and
$\ell$ of a resistive network (with matrix-valued parameters)
whose node admittance matrix is the Laplacian of the array's
interconnection graph $\Gamma$. From a graph-theoretic point of
view the nonsingularity of the effective conductance $\Gamma_{kl}$
may be interpreted to indicate that the pair of vertices
$(v_{k},\,v_{\ell})$ of the matrix-weighted graph $\Gamma$ are
connected. This therefore allows one to study connectivity of
vertices without employing paths; which is potentially useful
since defining a path, as mentioned above, is problematic for
matrix-weighted graphs. One may ask why our formulation is in
terms of effective conductance instead of the commoner effective
resistance, e.g., \cite{scardovi16}. The reason is that the
conductances we work with are matrix-valued and not necessarily
invertible. That is, since resistance is the inverse of
conductance, we would have run into certain difficulties had we
chosen to employ effective resistance instead. Potential
applications of generalized electrical circuits with matrix-valued
parameters seem to have so far gone unnoticed by the control
theorists. Notable exceptions are the works
\cite{barooah06,barooah07,barooah08} on the problem of estimation
over networks.

\section{Preliminaries and notation}

In this section we provide the formal definitions for the
observability of an array and the connectivity of an $n$-graph
through its Laplacian matrix. (The reader should be warned that
the term {\em $n$-graph} has appeared in the literature in
different meanings. In this paper it means a weighted graph, where
each pair of vertices is assigned an $n$-by-$n$ matrix.)

A pair $[(C_{ij})_{i,j=1}^{q},\,A]$ is meant to represent the
array of $q$ identical systems
\begin{subeqnarray}\label{eqn:array}
{\dot x}_{i}&=&Ax_{i}\\
y_{ij} &=&C_{ij}(x_{i}-x_{j})\,,\qquad\,i,\,j=1,\,2,\,\ldots,\,q
\end{subeqnarray}
where $x_{i}\in\Complex^{n}$ is the state of the $i$th system with
$A\in\Complex^{n\times n}$ and $y_{ij}\in\Complex^{m_{ij}}$ is the
$ij$th relative output with $C_{ij}\in\Complex^{m_{ij}\times n}$.
We let $C_{ii}=0$. In our paper we will solely be studying the
case $y_{ij}(t)\equiv 0$ for all $(i,\,j)$. Hence we suppose
$C_{ij}=C_{ji}$ without loss of generality. The generality is not
lost because if $C_{ij}\neq C_{ji}$ then we can always redefine
$C_{ij}^{\rm new}=C_{ji}^{\rm new}:=[C_{ij}^{T}\ C_{ji}^{T}]^{T}$;
and then $y^{\rm new}_{ij}(t)\equiv 0$ for all $(i,\,j)$ if and
only if $y_{ij}(t)\equiv 0$ for all $(i,\,j)$. The ordered
collection $(C_{ij})_{i,j=1}^{q}$ will sometimes be compactly
written as $(C_{ij})$ when there is no risk of ambiguity.

For each $(i,\,j)$ we denote by $W_{ij}$ the observability matrix
of the individual pair $[C_{ij},\,A]$. Namely,
\begin{eqnarray*}
W_{ij}=\left[\begin{array}{c}C_{ij}\\C_{ij}A\\\vdots\\C_{ij}A^{n-1}\end{array}\right]\,.
\end{eqnarray*}
The associated unobservable subspace is denoted by
$\U_{ij}\subset\Complex^{n}$. Recall that $\U_{ij}={\rm
null}\,W_{ij}$ and that $\U_{ij}$ is invariant under $A$. In
particular, $x_{i}(0)-x_{j}(0)\in\U_{ij}$ implies
$x_{i}(t)-x_{j}(t)\in\U_{ij}$ for all $t$ since
$x_{i}(t)-x_{j}(t)=e^{At}(x_{i}(0)-x_{j}(0))$. By
$\mu_{1},\,\mu_{2},\,\ldots,\,\mu_{m}$ ($1\leq m\leq n$) we denote
the distinct eigenvalues of $A$. By
$V_{\sigma}\in\Complex^{n\times n_{\sigma}}$,
$\sigma=1,\,2,\,\ldots,\,m$, we denote a full column rank matrix
satisfying ${\rm range}\,V_{\sigma}={\rm
null}\,[A-\mu_{\sigma}I_{n}]$, where $I_{n}$ is the $n$-by-$n$
identity matrix. Note that the columns of $V_{\sigma}$ are the
linearly independent eigenvectors of $A$ corresponding to the
eigenvalue $\mu_{\sigma}$. In particular, we have
$AV_{\sigma}=\mu_{\sigma}V_{\sigma}$. The below definition is what
this paper is all about.

\begin{definition}\label{def:obs}
An array $[(C_{ij}),\,A]$ is said to be {\em observable} if
\begin{eqnarray*}
y_{ij}(t)\equiv 0\ \mbox{for all}\ (i,\,j) \implies x_{i}(t)\equiv
x_{j}(t)\ \mbox{for all}\ (i,\,j)
\end{eqnarray*}
for all initial conditions
$x_{1}(0),\,x_{2}(0),\,\ldots,\,x_{q}(0)$.
\end{definition}

An {\em $n$-graph} $\Gamma=(\V,\,w)$ has a finite set of vertices
$\V=\{v_{1},\,v_{2},\,\ldots,\,v_{q}\}$ and a weight function
$w:\V\times \V\to\Complex^{n\times n}$ with the properties
\begin{itemize}
\item $w(v,\,v)=0$,\item $w(u,\,v)=w(v,\,u)$,\item
$w(u,\,v)=w(u,\,v)^{*}\geq 0$,
\end{itemize}
where $w(u,\,v)^*$ indicates the conjugate transpose of
$w(u,\,v)$. Let $G_{ij}=w(v_{i},\,v_{j})$. The $nq$-by-$nq$ matrix
\begin{eqnarray*}\label{eqn:Laplacian}
{\rm lap}\,\Gamma=\left[\begin{array}{cccc}\sum_{j} G_{1j}&-G_{12}&\cdots&-G_{1q}\\
-G_{21}&\sum_{j}G_{2j}&\cdots&-G_{2q}\\
\vdots&\vdots&\ddots&\vdots\\
-G_{q1}&-G_{q2}&\cdots&\sum_{j}G_{qj}
\end{array}\right]
\end{eqnarray*}
is called the {\em Laplacian} of $\Gamma$. Let $L={\rm
lap}\,\Gamma$. By construction the Laplacian is Hermitian, i.e.,
$L^{*}=L$, and enjoys some other desirable properties. Let
$\xi=[x_{1}^{T}\ x_{2}^{T}\ \cdots\ x_{q}^{T}]^{T}$ with
$x_{i}\in\Complex^{n}$ and define the synchronization subspace as
$\setS_{n}=\{\xi\in(\Complex^{n})^{q}:x_{i}=x_{j}\ \mbox{for all}\
(i,\,j)\}$. We see that ${\rm null}\,L\supset\setS_{n}$. Also,
since we can write
$\xi^{*}L\xi=\sum_{i>j}(x_{i}-x_{j})^{*}G_{ij}(x_{i}-x_{j})\geq
0$, the Laplacian is positive semidefinite. Therefore all its
eigenvalues are real and nonnegative, thanks to which the ordering
$\lambda_{1}(L)\leq\lambda_{2}(L)\leq\ldots\leq\lambda_{qn}(L)$ is
not meaningless. In the sequel, $\lambda_{k}(L)$ denotes the $k$th
smallest eigenvalue of $L$.

We denote by $\Gamma(H_{ij})_{i,j=1}^{q}$ (or by $\Gamma(H_{ij})$
when there is no risk of confusion) the $n$-graph of a collection
$(H_{ij})_{i,j=1}^{q}$ with
$H_{ji}=H_{ij}\in\Complex^{m_{ij}\times n}$ and $H_{ii}=0$. The
graph $\Gamma(H_{ij})$ has the vertex set
$\V=\{v_{1},\,v_{2},\,\ldots,\,v_{q}\}$ and its weight function
$w$ is such that $w(v_{i},\,v_{j})=H_{ij}^{*}H_{ij}$. Regarding
the array~\eqref{eqn:array} two graph constructions are
particularly important. One of them is the $n$-graph
$\Gamma(W_{ij})_{i,j=1}^{q}$ which we call the {\em
interconnection graph}. The other is the $n_{\sigma}$-graph
$\Gamma(C_{ij}V_{\sigma})_{i,j=1}^{q}$, called the {\em
eigengraph} corresponding to the eigenvalue $\mu_{\sigma}$.

In graph theory~\cite{chung92}, connectivity (in the classical
sense) is characterized by means of adjacency. A connected graph
is said to have a path between each pair of its vertices, where a
path is a sequence of adjacent vertices. For 1-graphs the
definition of adjacency is unequivocal: a pair of vertices
$(u,\,v)$ are adjacent if $w(u,\,v)>0$ and nonadjacent if
$w(u,\,v)=0$. (Adjacent vertices are said to have an edge between
them.) For $n$-graphs ($n\geq 2$) however, since we have the
in-between semidefinite case $w(u,\,v)\geq 0$, how to define
adjacency and, in turn, connectivity becomes a matter of choice.
For our purposes in this paper we (inevitably) abandon the concept
of adjacency altogether and define connectivity of a graph through
its Laplacian. Recall that a 1-graph $\Gamma$ is connected if and
only if $\lambda_{2}({\rm lap}\,\Gamma)>0$. Since this is an
equivalence result it can replace the definition of connectivity
for $1$-graphs. This substitute turns out to be much easier to
generalize than the standard definition that uses paths.

\begin{definition}
An $n$-graph $\Gamma$ is said to be {\em connected} if
$\lambda_{n+1}({\rm lap}\,\Gamma)>0$.
\end{definition}

The next three facts will find frequent use later in the paper.

\begin{lemma}\label{lem:one}
An $n$-graph $\Gamma$ is connected if and only if ${\rm
null}\,{\rm lap}\,\Gamma=\setS_{n}$.
\end{lemma}

\begin{proof}
Let $L={\rm lap}\,\Gamma$ denote the Laplacian. Suppose ${\rm
null}\,L=\setS_{n}$. By definition ${\rm dim}\,\setS_{n}=n$.
Therefore $L$ has $n$ linearly independent eigenvectors whose
eigenvalues are zero. Since $L^{*}=L$ this means that $L$ has
exactly $n$ eigenvalues at the origin. That all the eigenvalues of
$L$ are nonnegative then yields $\lambda_{n+1}(L)>0$. To show the
other direction this time we begin by letting
$\lambda_{n+1}(L)>0$. That is, $L$ has at most $n$ eigenvalues at
the origin. The property $L^{*}=L$ then implies that $L$ has at
most $n$ eigenvectors whose eigenvalues are zero. In other words,
${\rm dim}\,{\rm null}\,L\leq n$. This implies, in the light of
the facts ${\rm null}\,L\supset\setS_{n}$ and ${\rm
dim}\,\setS_{n}=n$, that ${\rm null}\,L=\setS_{n}$.
\end{proof}

\begin{lemma}\label{lem:two}
Consider the solutions $x_{i}(t)$ of the array~\eqref{eqn:array}.
Let $\xi(t)=[x_{1}(t)^{T}\ x_{2}(t)^{T}\ \cdots\
x_{q}(t)^{T}]^{T}$ and $L={\rm lap}\,\Gamma(W_{ij})$. We have
\begin{eqnarray*}
y_{ij}(t)\equiv 0\ \mbox{for all}\ (i,\,j) \quad
\Longleftrightarrow\quad L\xi(t) \equiv 0 \quad
\Longleftrightarrow\quad L\xi(0) = 0\,.
\end{eqnarray*}
\end{lemma}

\begin{proof}
Let $\zeta=[\rho_{1}^{T}\ \rho_{2}^{T}\ \ldots\ \rho_{q}^{T}]^{T}$
with $\rho_{i}\in\Complex^{n}$. Observe
$\sum_{i>j}\|W_{ij}(\rho_{i}-\rho_{j})\|^{2}
=\sum_{i>j}(\rho_{i}-\rho_{j})^{*}W_{ij}^{*}W_{ij}(\rho_{i}-\rho_{j})
=\zeta^{*}L\zeta$. Since $L^{*}=L\geq 0$ we also have
$\zeta^{*}L\zeta=0$ if and only if $L\zeta=0$. Recalling
$\U_{ij}={\rm null}\,W_{ij}$ we can now write
\begin{eqnarray*}
\rho_{i}-\rho_{j}\in\U_{ij}\ \mbox{for all}\ (i,\,j)
&\Longleftrightarrow& W_{ij}(\rho_{i}-\rho_{j})=0\ \mbox{for
all}\ (i,\,j)\\
&\Longleftrightarrow& \|W_{ij}(\rho_{i}-\rho_{j})\|^{2}=0\
\mbox{for all}\ (i,\,j)\\
&\Longleftrightarrow& \sum_{i>j}\|W_{ij}(\rho_{i}-\rho_{j})\|^{2}=
0\\
&\Longleftrightarrow& \zeta^{*}L\zeta=0\\
&\Longleftrightarrow& L\zeta=0\,.
\end{eqnarray*}
Therefore
\begin{eqnarray*}
y_{ij}(t)\equiv 0\ \mbox{for all}\ (i,\,j) &\Longleftrightarrow&
x_{i}(0)-x_{j}(0)\in\U_{ij}\ \mbox{for all}\ (i,\,j)\\
&\Longleftrightarrow& L\xi(0)=0\,.
\end{eqnarray*}
Recall that $x_{i}(0)-x_{j}(0)\in\U_{ij}$ implies
$x_{i}(t)-x_{j}(t)\in\U_{ij}$ for all $t$. Hence
\begin{eqnarray*}
y_{ij}(t)\equiv 0\ \mbox{for all}\ (i,\,j) &\Longleftrightarrow&
x_{i}(0)-x_{j}(0)\in\U_{ij}\ \mbox{for all}\ (i,\,j)\\
&\Longleftrightarrow& x_{i}(t)-x_{j}(t)\in\U_{ij}\ \mbox{for all}\
(i,\,j)\ \mbox{and}\ t\\
&\Longleftrightarrow& L\xi(t)\equiv 0
\end{eqnarray*}
which completes the proof.
\end{proof}

\begin{lemma}\label{lem:three}
Let $\sigma\in\{1,\,2,\,\ldots,\,m\}$ and $L={\rm
lap}\,\Gamma(W_{ij})$. The graph $\Gamma(C_{ij}V_{\sigma})$ is not
connected if and only if there exists a vector
$\zeta\notin\setS_{n}$ satisfying $\zeta\in{\rm
range}\,[I_{q}\otimes V_{\sigma}]\cap{\rm null}\,L$.
\end{lemma}

\begin{proof}
Given $\sigma$, let us suppose that $\Gamma(C_{ij}V_{\sigma})$ is
not connected. Since ${\rm
null}\,L_{\sigma}\supset\setS_{n_{\sigma}}$, by
Lemma~\ref{lem:one} there exists
$\eta\in(\Complex^{n_{\sigma}})^{q}$ that satisfies
$\eta\notin\setS_{n_{\sigma}}$ and $L_{\sigma}\eta=0$, where
$L_{\sigma}={\rm lap}\,\Gamma(C_{ij}V_{\sigma})$. Let us employ
the partition $\eta=[z_{1}^{T}\ z_{2}^{T}\ \cdots\ z_{q}^{T}]^{T}$
with $z_{i}\in\Complex^{n_{\sigma}}$ and define
$\zeta\in(\Complex^{n})^{q}$ as $\zeta=[I_{q}\otimes
V_{\sigma}]\eta$. That is, $\zeta=[(V_{\sigma}z_{1})^{T}\
(V_{\sigma}z_{2})^{T}\ \cdots\ (V_{\sigma}z_{q})^{T}]^{T}$.
Clearly, $\zeta\in{\rm range}\,[I_{q}\otimes V_{\sigma}]$. Since
$V_{\sigma}$ is full column rank, $\eta\notin\setS_{n_{\sigma}}$
yields $\zeta\notin\setS_{n}$. Lastly, we have to establish
$\zeta\in{\rm null}\,L$. Recall
$AV_{\sigma}=\mu_{\sigma}V_{\sigma}$. Therefore
\begin{eqnarray*}
V_{\sigma}^{*}W_{ij}^{*}W_{ij}V_{\sigma}
&=&V_{\sigma}^{*}\left(C_{ij}^{*}C_{ij}+A^{*}C_{ij}^{*}C_{ij}A+\cdots+A^{(n-1)*}C_{ij}^{*}C_{ij}A^{n-1}\right)V_{\sigma}\nonumber\\
&=&V_{\sigma}^{*}\left(C_{ij}^{*}C_{ij}+\mu_{\sigma}^{*}C_{ij}^{*}C_{ij}\mu_{\sigma}+\cdots+\mu_{\sigma}^{(n-1)*}C_{ij}^{*}C_{ij}\mu_{\sigma}^{n-1}\right)V_{\sigma}\nonumber\\
&=&\left(1+|\mu_{\sigma}|^2+\cdots+|\mu_{\sigma}|^{2(n-1)}\right)V_{\sigma}^{*}C_{ij}^{*}C_{ij}V_{\sigma}\,.
\end{eqnarray*}
Now we can proceed as
\begin{eqnarray*}
\zeta^{*}L\zeta
&=&\sum_{i>j}(V_{\sigma}z_{i}-V_{\sigma}z_{j})^{*}W_{ij}^{*}W_{ij}(V_{\sigma}z_{i}-V_{\sigma}z_{j})\\
&=&\sum_{i>j}(z_{i}-z_{j})^{*}V_{\sigma}^{*}W_{ij}^{*}W_{ij}V_{\sigma}(z_{i}-z_{j})\\
&=&\left(1+|\mu_{\sigma}|^2+\cdots+|\mu_{\sigma}|^{2(n-1)}\right)\sum_{i>j}(z_{i}-z_{j})^{*}V_{\sigma}^{*}C_{ij}^{*}C_{ij}V_{\sigma}(z_{i}-z_{j})\\
&=&\left(1+|\mu_{\sigma}|^2+\cdots+|\mu_{\sigma}|^{2(n-1)}\right)\eta^{*}L_{\sigma}\eta\\
&=&0\,.
\end{eqnarray*}
Since $L$ is Hermitian positive semidefinite, $\zeta^{*}L\zeta=0$
implies $L\zeta=0$, i.e., $\zeta\in{\rm null}\,L$.

To show the other direction, suppose that there exists
$\zeta\notin\setS_{n}$ satisfying $\zeta\in{\rm
range}\,[I_{q}\otimes V_{\sigma}]$ and $L\zeta=0$. Since
$\zeta\in{\rm range}\,[I_{q}\otimes V_{\sigma}]$ we can find
$\eta=[z_{1}^{T}\ z_{2}^{T}\ \cdots\ z_{q}^{T}]^{T}$ with
$z_{i}\in\Complex^{n_{\sigma}}$ satisfying $\zeta=[I_{q}\otimes
V_{\sigma}]\eta$. And since $V_{\sigma}$ is full column rank
$\zeta\notin\setS_{n}$ implies $\eta\notin\setS_{n_{\sigma}}$. Now
we turn the same wheels as in the first part, but in the opposite
direction.
\begin{eqnarray*}
\eta^{*}L_{\sigma}\eta&=&\sum_{i>j}(z_{i}-z_{j})^{*}V_{\sigma}^{*}C_{ij}^{*}C_{ij}V_{\sigma}(z_{i}-z_{j})\\
&=&\left(1+|\mu_{\sigma}|^2+\cdots+|\mu_{\sigma}|^{2(n-1)}\right)^{-1}\sum_{i>j}(z_{i}-z_{j})^{*}V_{\sigma}^{*}W_{ij}^{*}W_{ij}V_{\sigma}(z_{i}-z_{j})\\
&=&\left(1+|\mu_{\sigma}|^2+\cdots+|\mu_{\sigma}|^{2(n-1)}\right)^{-1}\sum_{i>j}(V_{\sigma}z_{i}-V_{\sigma}z_{j})^{*}W_{ij}^{*}W_{ij}(V_{\sigma}z_{i}-V_{\sigma}z_{j})\\
&=&\left(1+|\mu_{\sigma}|^2+\cdots+|\mu_{\sigma}|^{2(n-1)}\right)^{-1}\zeta^{*}L\zeta\\
&=&0\,.
\end{eqnarray*}
Since $L_{\sigma}$ is Hermitian positive semidefinite,
$\eta^{*}L_{\sigma}\eta=0$ implies $L_{\sigma}\eta=0$, i.e.,
$\eta\in{\rm null}\,L_{\sigma}$. This allows us to assert ${\rm
null}\,L_{\sigma}\neq\setS_{n_{\sigma}}$ because
$\eta\notin\setS_{n_{\sigma}}$. Then by Lemma~\ref{lem:one} the
graph $\Gamma(C_{ij}V_{\sigma})$ is not connected.
\end{proof}

\section{Observability and connectivity}

In this section we establish the equivalence between observability
and connectivity. Then we present a corollary on an interesting
special case followed by a relevant numerical example. We end the
section with a theorem on detectability. Below is our main result.

\begin{theorem}\label{thm:obs}
The following are equivalent.
\begin{enumerate}
\item The array $[(C_{ij}),\,A]$ is observable. \item The
interconnection graph $\Gamma(W_{ij})$ is connected. \item All the
eigengraphs
$\Gamma(C_{ij}V_{1}),\,\Gamma(C_{ij}V_{2}),\,\ldots,\,\Gamma(C_{ij}V_{m})$
are connected.
\end{enumerate}
\end{theorem}

\begin{proof}
{\em 1$\implies$2.} Suppose that $\Gamma(W_{ij})$ is not
connected. Hence ${\rm null}\,L\neq\setS_{n}$ by
Lemma~\ref{lem:one}, where $L={\rm lap}\,\Gamma(W_{ij})$. Since
${\rm null}\,L\supset\setS_{n}$ there must exist a vector
$\zeta\in(\Complex^{n})^{q}$ that satisfies both
$\zeta\notin\setS_{n}$ and $L\zeta=0$. Choose the initial
conditions $x_{i}(0)$ of the systems~\eqref{eqn:array} so as to
satisfy $[x_{1}(0)^{T}\ x_{2}(0)^{T}\ \cdots\
x_{q}(0)^{T}]^{T}=\zeta$. Then by Lemma~\ref{lem:two} we have
$y_{ij}(t)\equiv 0$ for all $(i,\,j)$. However there exists at
least one pair $(k,\,\ell)$ for which $x_{k}(t)\not\equiv
x_{\ell}(t)$ because $[x_{1}(0)^{T}\ x_{2}(0)^{T}\ \cdots\
x_{q}(0)^{T}]^{T}\notin\setS_{n}$. Hence the array
$[(C_{ij}),\,A]$ cannot be observable.

{\em 2$\implies$3.} Suppose that $\Gamma(C_{ij}V_{\sigma})$ is not
connected for some $\sigma\in\{1,\,2,\,\ldots,\,m\}$. Then by
Lemma~\ref{lem:three} there exists $\zeta\in(\Complex^{n})^{q}$
that satisfies $\zeta\notin\setS_{n}$ and $L\zeta=0$. That is,
${\rm null}\,L\neq\setS_{n}$. This implies by Lemma~\ref{lem:one}
that $\Gamma(W_{ij})$ is not connected.

{\em 3$\implies$1.} Suppose that the array $[(C_{ij}),\,A]$ is not
observable. Then we can find some initial conditions
$x_{1}(0),\,x_{2}(0),\,\ldots,\,x_{q}(0)$ for which the solutions
$x_{i}(t)$ of the systems~\eqref{eqn:array} yield
\begin{eqnarray*}
y_{ij}(t)\equiv 0\ \mbox{for all}\ (i,\,j)\ \mbox{and}\
x_{k}(t)\not\equiv x_{\ell}(t)\ \mbox{for some}\ (k,\,\ell)\,.
\end{eqnarray*}
Let $\xi(t)=[x_{1}(t)^{T}\ x_{2}(t)^{T}\ \cdots\
x_{q}(t)^{T}]^{T}$. By Lemma~\ref{lem:two} we have $L\xi(0)=0$
because $y_{ij}(t)\equiv 0$ for all $(i,\,j)$. We also have
$x_{k}(0)\neq x_{\ell}(0)$ because $x_{k}(t)\not\equiv
x_{\ell}(t)$. Hence $\xi(0)\notin\setS_{n}$. Combining $L\xi(0)=0$
and $\xi(0)\notin\setS_{n}$ (in the light of ${\rm
null}\,L\supset\setS_{n}$) implies that ${\rm null}\,L$ is a
strict superset of $\setS_{n}$. Let ${\rm dim}\,{\rm
null}\,L=n+p$. (Note that $p\geq 1$.) Let
$S\in\Complex^{(nq)\times n}$ and $U\in\Complex^{(nq)\times p}$ be
two full column rank matrices satisfying ${\rm
range}\,S=\setS_{n}$ and ${\rm range}\,[S\ \,U]={\rm null}\,L$.
Recall that the unobservable subspaces $\U_{ij}={\rm
null}\,W_{ij}$ are invariant with respect to the matrix $A$. As a
consequence ${\rm null}\,L$ is invariant with respect to the
matrix $[I_{q}\otimes A]$. To see that let $\zeta=[\rho_{1}^{T}\
\rho_{2}^{T}\ \cdots\ \rho_{q}^{T}]^{T}$ with
$\rho_{i}\in\Complex^{n}$. We can write
\begin{eqnarray*}
\zeta\in{\rm null}\,L &\implies& \sum_{i>j}\|W_{ij}(\rho_{i}-\rho_{j})\|^2=\zeta^{*}L\zeta=0\\
&\implies&(\rho_{i}-\rho_{j})\in{\rm null}\,W_{ij}\ \mbox{for all}\ (i,\,j)\\
&\implies&A(\rho_{i}-\rho_{j})\in{\rm null}\,W_{ij}\ \mbox{for all}\ (i,\,j)\\
&\implies&\sum_{i>j}\|W_{ij}(A\rho_{i}-A\rho_{j})\|^2=\zeta^{*}[I_{q}\otimes
A]^{*}L[I_{q}\otimes A]\zeta=0\\
&\implies&[I_{q}\otimes A]\zeta\in{\rm null}\,L
\end{eqnarray*}
where for the last implication we use the fact that $L$ is
Hermitian positive semidefinite. Now, due to invariance, there
have to exist matrices $\Omega\in\Complex^{n\times p}$ and
$\Lambda\in\Complex^{p\times p}$ that satisfy
\begin{eqnarray*}
[I_{q}\otimes A]U=S\Omega+U\Lambda\,.
\end{eqnarray*}
Let $f\in\Complex^{p}$ be an eigenvector of $\Lambda$ with
eigenvalue $\lambda\in\Complex$, i.e., $\Lambda f=\lambda f$.
Also, let $\zeta_{1}=Uf$ and $\zeta_{2}=S\Omega f$. Note that
$\zeta_{1}\notin\setS_{n}$ and $\zeta_{2}\in\setS_{n}$. Now we can
write
\begin{eqnarray*}
\left([I_{q}\otimes A]-\lambda I_{nq}\right)\zeta_{1}
&=&[I_{q}\otimes A]Uf-\lambda Uf\\
&=&[I_{q}\otimes A]Uf-U\Lambda f\\
&=&\left([I_{q}\otimes A]U-U\Lambda\right)f\\
&=&S\Omega f\\
&=&\zeta_{2}\,.
\end{eqnarray*}
Let us employ the partitions $\zeta_{1}=[{\tilde \rho}_{1}^{T}\
{\tilde \rho}_{2}^{T}\ \cdots\ {\tilde \rho}_{q}^{T}]^{T}$ with
${\tilde \rho}_{i}\in\Complex^{n}$ and $\zeta_{2}=[{\bar
\rho}^{T}\ {\bar \rho}^{T}\ \cdots\ {\bar \rho}^{T}]^{T}$ with
${\bar \rho}\in\Complex^{n}$. Then $\left([I_{q}\otimes A]-\lambda
I_{nq}\right)\zeta_{1}=\zeta_{2}$ yields $(A-\lambda I_{n}){\tilde
\rho}_{i}={\bar \rho}$ for all $i\in\{1,\,2,\,\ldots,\,q\}$.
Choose an arbitrary index $a\in\{1,\,2,\,\ldots,\,q\}$ and define
$\zeta_{3}=[{\tilde \rho}_{a}^{T}\ {\tilde \rho}_{a}^{T}\ \cdots\
{\tilde \rho}_{a}^{T}]^{T}$ and $\zeta_{4}=\zeta_{1}-\zeta_{3}$.
Note that $\zeta_{3}\in\setS_{n}$ and $\zeta_{4}\notin\setS_{n}$.
Moreover, since both $\zeta_{1}$ and $\zeta_{3}$ belong to ${\rm
null}\,L$, we have $L\zeta_{4}=0$. Now observe
\begin{eqnarray*}
\left([I_{q}\otimes A]-\lambda
I_{nq}\right)\zeta_{4}&=&\left([I_{q}\otimes
A]-\lambda I_{nq}\right)(\zeta_{1}-\zeta_{3})\\
&=&\zeta_{2}-\left([I_{q}\otimes
A]-\lambda I_{nq}\right)\zeta_{3}\\
&=&\left[\begin{array}{c}{\bar \rho}\\\vdots\\{\bar
\rho}\end{array}\right]-\left[\begin{array}{c}(A-\lambda
I_{n}){\tilde \rho}_{a}\\ \vdots\\ (A-\lambda I_{n}){\tilde \rho}_{a}\end{array}\right]\\
&=&0\,.
\end{eqnarray*}
Let $[{\hat\rho}_{1}^{T}\ {\hat\rho}_{2}^{T}\ \cdots\
{\hat\rho}_{q}^{T}]^{T}=\zeta_{4}$. We can write
\begin{eqnarray*}
\left[\begin{array}{c}(A-\lambda I_{n}){\hat\rho}_{1}\\
\vdots\\ (A-\lambda
I_{n}){\hat\rho}_{q}\end{array}\right]=\left([I_{q}\otimes
A]-\lambda I_{nq}\right)\zeta_{4}=0\,.
\end{eqnarray*}
Therefore $\lambda=\mu_{\sigma}$ for some
$\sigma\in\{1,\,2,\,\ldots,\,m\}$ and every nonzero
${\hat\rho}_{i}$ is an eigenvector of $A$. In particular, for each
${\hat\rho}_{i}$ there uniquely exists
$z_{i}\in\Complex^{n_{\sigma}}$ such that
${\hat\rho}_{i}=V_{\sigma}z_{i}$. By stacking these $z_{i}$ into
$\eta=[z_{1}^{T}\ z_{2}^{T}\ \cdots\ z_{q}^{T}]^{T}$ we have
$\zeta_{4}=[I_{q}\otimes V_{\sigma}]\eta$. Recall that we have
already obtained $\zeta_{4}\notin\setS_{n}$ and $L\zeta_{4}=0$.
Hence Lemma~\ref{lem:three} assures us that the eigengraph
$\Gamma(C_{ij}V_{\sigma})$ is not connected.
\end{proof}

\vspace{0.12in}

Theorem~\ref{thm:obs} has an interesting implication concerning
$1$-graphs. Let $d_{A}(s)$ and $m_{A}(s)$ respectively denote the
characteristic polynomial and the minimal polynomial of the matrix
$A$. Note that if each $V_{\sigma}$ consists of a single column,
i.e., each eigenvalue $\mu_{\sigma}$ has a unique (up to a
scaling) eigenvector, then all the eigengraphs
$\Gamma(C_{ij}V_{1}),\,\Gamma(C_{ij}V_{2}),\,\ldots,\,\Gamma(C_{ij}V_{m})$
become $1$-graphs. A sufficient condition for this is that the
eigenvalues $\mu_{1},\,\mu_{2},\,\ldots,\,\mu_{m}$ are all simple,
i.e., $m=n$. More generally:

\begin{corollary}\label{cor:obs}
Suppose $m_{A}(s)=d_{A}(s)$. Then the array $[(C_{ij}),\,A]$ is
observable if and only if all the $1$-graphs
$\Gamma(C_{ij}V_{1}),\,\Gamma(C_{ij}V_{2}),\,\ldots,\,\Gamma(C_{ij}V_{m})$
are connected.
\end{corollary}

{\em An example.} Consider the array $[(C_{ij})_{i,j=1}^{4},\,A]$
with
\begin{eqnarray*}
A=\left[\begin{array}{rrrrrr}
 0 &    1 &   -7 &  -14 &   21 &   31\\
 1 &    1 &    1 &    3 &   -7 &  -11\\
 3 &    6 &  -28 &  -43 &    7 &    5\\
-2 &   -4 &   18 &   28 &   -7 &   -7\\
-2 &   -4 &   -2 &    1 &  -32 &  -49\\
 1 &    2 &    3 &    2 &   20 &   31
\end{array}\right]
\end{eqnarray*}
and
\begin{eqnarray*}
\begin{array}{rrrrrrrr}
C_{12}&=&\big[\ 2&3&8&12&6&10\ \big]\,,\\
&&&&&&&\\
C_{34}&=&\big[\ 4&6&6&10&6&9\ \big]\,,
\end{array}\qquad
\begin{array}{rrrrrrrr}
C_{23}&=&\big[\ 2&3&4&6&6&9\ \big]\,,\\
&&&&&&&\\
C_{41}&=&\big[\ 1&2&6&9&4&7\ \big]\,.\end{array}
\end{eqnarray*}
The matrices $C_{13}$ and $C_{24}$ are zero. (Recall
$C_{ij}=C_{ji}$ and $C_{ii}=0$.) The characteristic polynomial of
$A$ reads $d_{A}(s)=s^{6}-s^{2}=s^{2}(s-1)(s+1)(s-j)(s+j)$. That
is, $A$ has $m=5$ distinct eigenvalues:
$\mu_{1}=0,\,\mu_{2}=1,\,\mu_{3}=-1,\,\mu_{4,5}=\pm j$. The
eigenvalue at the origin is repeated, yet ${\rm dim}\,{\rm
null}\,A=1$. Therefore there is a single eigenvector corresponding
to $\mu_{1}$. Consequently we have $m_{A}(s)=d_{A}(s)$. The
matrices (or, in this case, vectors) $V_{\sigma}$ corresponding to
the eigenvalues $\mu_{\sigma}$ are given below.
\begin{eqnarray*}
V_{1}=\left[
\begin{array}{r}
-5\\ 2\\ -4\\ 3\\ 5\\ -3
\end{array}\right]\,,\quad
V_{2}=\left[
\begin{array}{r}
1\\ -1\\ -5\\ 3\\ -4\\ 3
\end{array}\right]\,,\quad
V_{3}=\left[
\begin{array}{r}
-2\\ 1\\ 2\\ -1\\ 3\\ -2
\end{array}\right]\,,\quad
V_{4,5}=\left[
\begin{array}{r}
-17\\ 4\\ -19\\ 14\\ 22\\ -13
\end{array}\right]\pm j\left[
\begin{array}{r}
0\\ 1\\ 8\\ -5\\ -3\\ 1
\end{array}\right]\,.
\end{eqnarray*}
Now let us determine whether the array $[(C_{ij}),\,A]$ is
observable or not. Thanks to Corollary~\ref{cor:obs} we can do
this by checking the connectivities of the $1$-graphs
$\Gamma(C_{ij}V_{1}),\,\Gamma(C_{ij}V_{2}),\,\ldots,\,\Gamma(C_{ij}V_{5})$.
A pleasant thing about a $1$-graph is that its connectivity can be
read from its visual representation. In this universal picture
every vertex is represented by a dot and a line (called an edge)
is drawn connecting a pair of vertices $(v_{i},\,v_{j})$ if the
value of the weight function (which is scalar for a 1-graph) is
positive, i.e., $w(v_{i},\,v_{j})>0$. Then the graph is connected
if we can reach from any dot to any other dot by tracing the
lines. Our 1-graphs are shown in Fig.~\ref{fig:onegraphs}.
Clearly, all of them are connected. Hence we conclude that the
array $[(C_{ij}),\,A]$ is observable.
\begin{figure}[h]
\begin{center}
\includegraphics[scale=0.55]{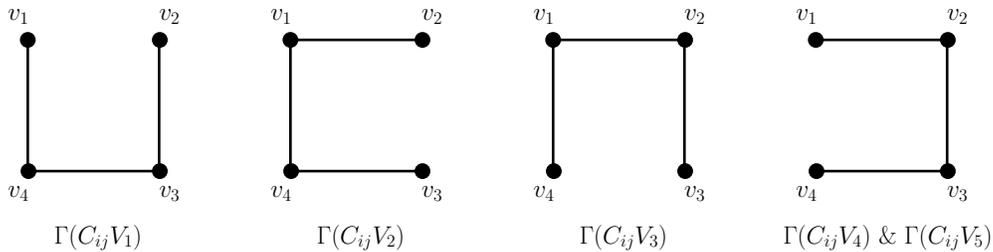}
\caption{Connected eigengraphs.}\label{fig:onegraphs}
\end{center}
\end{figure}
Note that for some nonzero $C_{ij}$ certain edges
$(v_{i},\,v_{j})$ are missing. For instance, for the graph
$\Gamma(C_{ij}V_{1})$ the edge $(v_{1},\,v_{2})$ is absent despite
$C_{12}\neq 0$. The reason is that $\mu_{1}=0$ is an unobservable
eigenvalue for the pair $[C_{12},\,A]$. In particular, we have
$C_{12}V_{1}=0$, i.e., the eigenvector $V_{1}$ belongs to the null
space of $C_{12}$. Interestingly, the intersection of the graphs
in Fig.~\ref{fig:onegraphs} yields an empty set of edges. This is
because there is not a single pair $[C_{ij},\,A]$ that is
observable. Still, that does not prevent the array
$[(C_{ij}),\,A]$ from being observable.

A few words on a practical issue. As is well known, sometimes the
system designer may have to settle upon less than observability,
where $y_{ij}(t)\equiv 0$ need not imply the desired
$x_{i}(t)\equiv x_{j}(t)$, but rather guarantees only
$\|x_{i}(t)-x_{j}(t)\|\to 0$ as $t\to\infty$. This suggests to
slacken Definition~\ref{def:obs} a bit.
\begin{definition}\label{def:detect}
An array $[(C_{ij}),\,A]$ is said to be {\em detectable} if
\begin{eqnarray*}
y_{ij}(t)\equiv 0\ \mbox{for all}\ (i,\,j) \implies
\|x_{i}(t)-x_{j}(t)\|\to 0\ \mbox{for all}\ (i,\,j)
\end{eqnarray*}
for all initial conditions
$x_{1}(0),\,x_{2}(0),\,\ldots,\,x_{q}(0)$.
\end{definition}

Since (for a continuous-time system) the terms in the solution
related to the eigenvalues of $A$ that are on the open left
half-plane will die out eventually, detectability is assured if
the eigenvalues on the closed right half-plane are observable.
More formally:

\begin{theorem}\label{thm:det}
The array $[(C_{ij}),\,A]$ is detectable if and only if all the
eigengraphs $\Gamma(C_{ij}V_{\sigma})$ with ${\rm
Re}\,\mu_{\sigma}\geq 0$ are connected.
\end{theorem}

\begin{proof}
{\em Con.$\implies$Det.} Suppose that $[(C_{ij}),\,A]$ is not
detectable. Then we can find some initial conditions
$x_{1}(0),\,x_{2}(0),\,\ldots,\,x_{q}(0)$ for which the solutions
$x_{i}(t)$ of the systems~\eqref{eqn:array} yield
\begin{eqnarray*}
y_{ij}(t)\equiv 0\ \mbox{for all}\ (i,\,j)\ \mbox{and}\
\|x_{k}(t)-x_{\ell}(t)\|\not\to 0
\end{eqnarray*}
for some pair $(k,\,\ell)$. Let $\xi(t)=[x_{1}(t)^{T}\
x_{2}(t)^{T}\ \cdots\ x_{q}(t)^{T}]^{T}$. Note that each
$x_{i}(t)$ solves ${\dot x}_{i}=Ax_{i}$. Also, recall that
$\mu_{1},\,\mu_{1},\,\ldots,\,\mu_{m}$ denote the distinct
eigenvalues of $A$. Therefore $\xi(t)$ solves
${\dot\xi}=[I_{q}\otimes A]\xi$ and enjoys the structure
\begin{eqnarray*}
\xi(t)=p_{1}(t)e^{\mu_{1}t}+p_{2}(t)e^{\mu_{2}t}+\cdots+p_{m}(t)e^{\mu_{m}t}
\end{eqnarray*}
for some polynomials $p_{1}(t),\,p_{2}(t),\,\ldots,\,p_{m}(t)$
whose coefficients are vectors in $(\Complex^{n})^{q}$. Let
$L={\rm lap}\,\Gamma(W_{ij})$ and observe
$[(e_{k}-e_{\ell})\otimes I_{n}]^{T}\xi(t)=x_{k}(t)-x_{\ell}(t)$.
By Lemma~\ref{lem:two} we have $L\xi(t)\equiv 0$ because
$y_{ij}(t)\equiv 0$ for all $(i,\,j)$. We also have
$[(e_{k}-e_{\ell})\otimes I_{n}]^{T}\xi(t)\not\to 0$ because
$\|x_{k}(t)-x_{\ell}(t)\|\not\to 0$. Let us here make a few
observations. Since $\mu_{1},\,\mu_{2},\,\ldots,\,\mu_{m}$ are
distinct, the collection of mappings $\{t\mapsto
p_{\sigma}(t)e^{\mu_{\sigma}t}:p_{\sigma}(t)\not\equiv
0,\,\sigma=1,\,2,\,\ldots,\,m\}$ are linearly independent.
Therefore $L\xi(t)\equiv 0$ implies
\begin{eqnarray}\label{eqn:hungry}
Lp_{\sigma}(t)e^{\mu_{\sigma}t}\equiv 0
\end{eqnarray}
for all $\sigma$. Moreover, $[(e_{k}-e_{\ell})\otimes
I_{n}]^{T}\xi(t)\not\to 0$ implies
\begin{eqnarray}\label{eqn:thirsty}
[(e_{k}-e_{\ell})\otimes
I_{n}]^{T}p_{\sigma}(t)e^{\mu_{\sigma}t}\not\to 0
\end{eqnarray}
for some $\sigma$. Finally, ${\dot \xi}(t)=[I_{q}\otimes A]\xi(t)$
implies
\begin{eqnarray}\label{eqn:tired}
\frac{d}{dt}\left\{p_{\sigma}(t)e^{\mu_{\sigma}t}\right\}=[I_{q}\otimes
A]p_{\sigma}(t)e^{\mu_{\sigma}t}
\end{eqnarray}
for all $\sigma$. Let us now fix an index
$\sigma\in\{1,\,2,\,\ldots,\,m\}$ that satisfies
\eqref{eqn:thirsty}. Clearly, we have ${\rm Re}\,\mu_{\sigma}\geq
0$. Let
$p_{\sigma}(t)=\vartheta_{r}t^{r}+\cdots+\vartheta_{1}t+\vartheta_{0}$
with
$\vartheta_{0},\,\vartheta_{1},\,\ldots,\,\vartheta_{r}\in(\Complex^{n})^{q}$
and $\vartheta_{r}\neq 0$.  By \eqref{eqn:tired} we can write
$p_{\sigma}(t)e^{\mu_{\sigma}t}=[I_{q}\otimes
e^{At}]\vartheta_{0}$. This implies
\begin{eqnarray*}
[(e_{k}-e_{\ell})\otimes I_{n}]^{T}\vartheta_{0}\neq 0
\end{eqnarray*}
for otherwise ($[(e_{k}-e_{\ell})\otimes
I_{n}]^{T}\vartheta_{0}=0$) we would have had
\begin{eqnarray*}
[(e_{k}-e_{\ell})\otimes
I_{n}]^{T}p_{\sigma}(t)e^{\mu_{\sigma}t}&=&
[(e_{k}-e_{\ell})\otimes I_{n}]^{T}[I_{q}\otimes
e^{At}]\vartheta_{0}\\
&=& [(e_{k}-e_{\ell})^{T}\otimes e^{At}]\vartheta_{0}\\
&=& e^{At}[(e_{k}-e_{\ell})\otimes
I_{n}]^{T}\vartheta_{0}\\
&=&0
\end{eqnarray*}
which contradicts \eqref{eqn:thirsty}. Note that the mappings
$t\mapsto\vartheta_{0}e^{\mu_{\sigma}t},\,t\mapsto\vartheta_{1}te^{\mu_{\sigma}t},\,\ldots,\,t\mapsto\vartheta_{r}t^{r}e^{\mu_{\sigma}t}$
are linearly independent. Therefore \eqref{eqn:hungry} yields
$L\vartheta_{\nu}t^{\nu}e^{\mu_{\sigma}t}\equiv 0$ for all
$\nu\in\{0,\,1,\,\ldots,\,r\}$. Consequently, $L\vartheta_{\nu}=0$
for all $\nu\in\{0,\,1,\,\ldots,\,r\}$.

Since
$t\mapsto(\vartheta_{r}t^{r}+\cdots+\vartheta_{1}t+\vartheta_{0})e^{\mu_{\sigma}t}$
is a solution of ${\dot\xi}=[I_{q}\otimes A]\xi$ we have the
following chain
\begin{eqnarray*}
([I_{q}\otimes A]-\mu_{\sigma}I_{nq})\vartheta_{0}&=&\vartheta_{1}\\
([I_{q}\otimes A]-\mu_{\sigma}I_{nq})\vartheta_{1}&=&2\vartheta_{2}\\
&\vdots&\\
([I_{q}\otimes A]-\mu_{\sigma}I_{nq})\vartheta_{r-1}&=&r\vartheta_{r}\\
([I_{q}\otimes A]-\mu_{\sigma}I_{nq})\vartheta_{r}&=&0\,.
\end{eqnarray*}
Let us now fix an index $\nu\in\{0,\,1,\,\ldots,\,r\}$ that
satisfies $\vartheta_{\nu}\notin\setS_{n}$ and $([I_{q}\otimes
A]-\mu_{\sigma}I_{nq})\vartheta_{\nu}\in\setS_{n}$. Such $\nu$
should exist because $\vartheta_{0}\notin\setS_{n}$ (thanks to
$[(e_{k}-e_{\ell})\otimes I_{n}]^{T}\vartheta_{0}\neq 0$) and
$([I_{q}\otimes
A]-\mu_{\sigma}I_{nq})\vartheta_{r}=0\in\setS_{n}$. Let
$\zeta_{1}=[{\tilde \rho}_{1}^{T}\ {\tilde \rho}_{2}^{T}\ \cdots\
{\tilde \rho}_{q}^{T}]^{T}=\vartheta_{\nu}$ with ${\tilde
\rho}_{i}\in\Complex^{n}$ and $\zeta_{2}=[{\bar \rho}^{T}\ {\bar
\rho}^{T}\ \cdots\ {\bar \rho}^{T}]^{T}=([I_{q}\otimes
A]-\mu_{\sigma}I_{nq})\vartheta_{\nu}$ with ${\bar
\rho}\in\Complex^{n}$. Then $\left([I_{q}\otimes A]-\mu_{\sigma}
I_{nq}\right)\zeta_{1}=\zeta_{2}$ yields $(A-\mu_{\sigma}
I_{n}){\tilde \rho}_{i}={\bar \rho}$ for all
$i\in\{1,\,2,\,\ldots,\,q\}$. Choose an arbitrary index
$a\in\{1,\,2,\,\ldots,\,q\}$ and define $\zeta_{3}=[{\tilde
\rho}_{a}^{T}\ {\tilde \rho}_{a}^{T}\ \cdots\ {\tilde
\rho}_{a}^{T}]^{T}$ and $\zeta_{4}=\zeta_{1}-\zeta_{3}$. Note that
$\zeta_{3}\in\setS_{n}$ and $\zeta_{4}\notin\setS_{n}$. Moreover,
since both $\zeta_{1}$ and $\zeta_{3}$ belong to ${\rm null}\,L$,
we have $L\zeta_{4}=0$. Observe
\begin{eqnarray*}
\left([I_{q}\otimes A]-\mu_{\sigma}
I_{nq}\right)\zeta_{4}&=&\left([I_{q}\otimes
A]-\mu_{\sigma} I_{nq}\right)(\zeta_{1}-\zeta_{3})\\
&=&\zeta_{2}-\left([I_{q}\otimes
A]-\mu_{\sigma} I_{nq}\right)\zeta_{3}\\
&=&\left[\begin{array}{c}{\bar \rho}\\\vdots\\{\bar
\rho}\end{array}\right]-\left[\begin{array}{c}(A-\mu_{\sigma}
I_{n}){\tilde \rho}_{a}\\ \vdots\\ (A-\mu_{\sigma} I_{n}){\tilde \rho}_{a}\end{array}\right]\\
&=&0\,.
\end{eqnarray*}
Let $[\rho_{1}^{T}\ \rho_{2}^{T}\ \cdots\
\rho_{q}^{T}]^{T}=\zeta_{4}$. We can write
\begin{eqnarray*}
\left[\begin{array}{c}(A-\mu_{\sigma} I_{n})\rho_{1}\\
\vdots\\ (A-\mu_{\sigma}
I_{n})\rho_{q}\end{array}\right]=\left([I_{q}\otimes
A]-\mu_{\sigma} I_{nq}\right)\zeta_{4}=0\,.
\end{eqnarray*}
Therefore every nonzero $\rho_{i}$ is an eigenvector of $A$ with
eigenvalue $\mu_{\sigma}$. In particular, for each $\rho_{i}$
there uniquely exists $z_{i}\in\Complex^{n_{\sigma}}$ such that
$\rho_{i}=V_{\sigma}z_{i}$. By stacking these $z_{i}$ into
$\eta=[z_{1}^{T}\ z_{2}^{T}\ \cdots\ z_{q}^{T}]^{T}$ we have
$\zeta_{4}=[I_{q}\otimes V_{\sigma}]\eta$. Recall that we have
already obtained $\zeta_{4}\notin\setS_{n}$ and $L\zeta_{4}=0$.
Hence Lemma~\ref{lem:three} assures us that the eigengraph
$\Gamma(C_{ij}V_{\sigma})$ is not connected.

{\em Det.$\implies$Con.} Suppose that for some
$\sigma\in\{1,\,2,\,\ldots,\,m\}$ the graph
$\Gamma(C_{ij}V_{\sigma})$ is not connected and ${\rm
Re}\,\mu_{\sigma}\geq 0$. By Lemma~\ref{lem:three} we know that we
can find $\zeta\in(\Complex^{n})^{q}$ which can be written as
$\zeta=[I_{q}\otimes V_{\sigma}]\eta$ for some
$\eta\in(\Complex^{n_{\sigma}})^{q}$ while satisfying $L\zeta=0$
and $\zeta\notin\setS_{n}$. Choose the initial conditions
$x_{i}(0)$ of the systems~\eqref{eqn:array} so as to satisfy
$[x_{1}(0)^{T}\ x_{2}(0)^{T}\ \cdots\ x_{q}(0)^{T}]^{T}=\zeta$.
Note that
\begin{eqnarray*}
\left[\begin{array}{c}Ax_{1}(0)\\\vdots\\Ax_{q}(0)\end{array}\right]&=&[I_{q}\otimes
A]\zeta=[I_{q}\otimes A][I_{q}\otimes
V_{\sigma}]\eta=[I_{q}\otimes
AV_{\sigma}]\eta\\
&=&[I_{q}\otimes
\mu_{\sigma}V_{\sigma}]\eta=\mu_{\sigma}[I_{q}\otimes
V_{\sigma}]\eta
=\mu_{\sigma}\zeta\\
&=&\left[\begin{array}{c}\mu_{\sigma}x_{1}(0)\\\vdots\\\mu_{\sigma}x_{q}(0)\end{array}\right]\,.
\end{eqnarray*}
That is, each nonzero $x_{i}(0)$ is an eigenvector of $A$ with
eigenvalue $\mu_{\sigma}$. Therefore we have
$x_{i}(t)=x_{i}(0)e^{\mu_{\sigma}t}$ for all
$i\in\{1,\,2,\,\ldots,\,q\}$. Since $L\zeta=0$ the solutions
$x_{i}(t)$ yield by Lemma~\ref{lem:two} $y_{ij}(t)\equiv 0$ for
all $(i,\,j)$. However, there exists some pair $(k,\,\ell)$ for
which $x_{k}(0)\neq x_{\ell}(0)$ because $[x_{1}(0)^{T}\
x_{2}(0)^{T}\ \cdots\ x_{q}(0)^{T}]^{T}=\zeta\notin\setS_{n}$.
Moreover, we can write
\begin{eqnarray*}
\|x_{k}(t)-x_{\ell}(t)\|=\|x_{k}(0)e^{\mu_{\sigma}t}-x_{\ell}(0)e^{\mu_{\sigma}t}\|=|e^{\mu_{\sigma}t}|\cdot\|x_{k}(0)-x_{\ell}(0)\|\not\to
0
\end{eqnarray*}
because ${\rm Re}\,\mu_{\sigma}\geq 0$. Therefore the array
$[(C_{ij}),\,A]$ cannot be detectable.
\end{proof}

\section{Pairwise observability and effective conductance}

Hitherto, regarding the array~\eqref{eqn:array}, we have focused
solely on the total synchronization, i.e., $x_{i}(t)\equiv
x_{j}(t)$ for all pairs $(i,\,j)$. Henceforth we consider partial
synchronization, where certain pairs of systems are possibly out
of synchrony. The reason that this problem is worth tackling is
threefold. First, it puts the previous analysis on a more complete
footing. Second, for applications where total synchronization is
desired but not achieved, it is meaningful to want to determine
how far off we are from the goal. Third, it is not difficult to
imagine situations where total synchronization is not desired for
reasons of security. For instance, in a secure communication
scenario there may be an array with many identical systems among
which there is a pair of distant units that are desired to
synchronize without synchronizing with their immediate neighbors.
These motivate the following definition.

\begin{definition}\label{def:klobs}
An array $[(C_{ij}),\,A]$ is said to be $(k,\,\ell)$-{\em
observable} if
\begin{eqnarray*}
y_{ij}(t)\equiv 0\ \mbox{for all}\ (i,\,j) \implies x_{k}(t)\equiv
x_{\ell}(t)
\end{eqnarray*}
for all initial conditions
$x_{1}(0),\,x_{2}(0),\,\ldots,\,x_{q}(0)$.
\end{definition}
And the sister definition reads:
\begin{definition}\label{def:kldet}
An array $[(C_{ij}),\,A]$ is said to be $(k,\,\ell)$-{\em
detectable} if
\begin{eqnarray*}
y_{ij}(t)\equiv 0\ \mbox{for all}\ (i,\,j) \implies \|x_{k}(t)-
x_{\ell}(t)\|\to 0
\end{eqnarray*}
for all initial conditions
$x_{1}(0),\,x_{2}(0),\,\ldots,\,x_{q}(0)$.
\end{definition}

As before, where we studied the observability of an array by means
of the connectivity of interconnection graph, we will again
approach the problem from graphic angle. Recall that a 1-graph is
connected when the null space of its Laplacian is spanned by the
vector of all ones, which led us to the generalization stated in
Lemma~\ref{lem:one}. Likewise, a pair of vertices
$(v_{k},\,v_{\ell})$ of a 1-graph is connected if any vector that
belongs to the null space of the Laplacian is with identical $k$th
and $\ell$th entries. We now obtain the natural generalization of
pairwise connectivity for $n$-graphs. Let $\Gamma$ be an $n$-graph
with $q$ vertices. Let $e_{i}\in\Complex^{q}$ be the unit vector
with $i$th entry one and the remaining entries zero. For $n=1$ the
connectedness of the pair $(v_{k},\,v_{\ell})$ is equivalent to
the condition $(e_{k}-e_{\ell})\in({\rm null}\,{\rm
lap}\,\Gamma)^{\perp}$. This suggests:

\begin{definition}
An $n$-graph $\Gamma$ is said to be {\em $(k,\,\ell)$-connected}
if ${\rm range}\,[(e_{k}-e_{\ell})\otimes I_{n}]\subset({\rm
null}\,{\rm lap}\,\Gamma)^{\perp}$.
\end{definition}

It may seem reasonable to expect that Theorem~\ref{thm:obs} and
Theorem~\ref{thm:det} of the previous section can be effortlessly
converted into ``pairwise'' statements by simply replacing the
words {\em observable, detectable, connected} therein with {\em
$(k,\,\ell)$-observable, $(k,\,\ell)$-detectable,
$(k,\,\ell)$-connected}, respectively. Surprisingly enough this is
not the case; certain associations disappear in the pairwise
domain. In particular, an array that is not
$(k,\,\ell)$-observable may still have all its eigengraphs
$(k,\,\ell)$-connected. We now proceed by establishing the
remaining links. Then we provide evidence (counterexample) for the
missing implications.

\begin{theorem}\label{thm:klobs}
The array $[(C_{ij}),\,A]$ is $(k,\,\ell)$-observable if and only
if the $n$-graph $\Gamma(W_{ij})$ is $(k,\,\ell)$-connected.
\end{theorem}

\begin{proof}
{\em Con.$\implies$Obs.} Suppose that $\Gamma(W_{ij})$ is
$(k,\,\ell)$-connected. Consider the solutions $x_{i}(t)$ of the
systems~\eqref{eqn:array}. Let $L={\rm lap}\,\Gamma(W_{ij})$ be
the Laplacian and $\xi(t)=[x_{1}^{T}(t)\ x_{2}^{T}(t)\ \cdots\
x_{q}^{T}(t)]^{T}$. Since $\Gamma(W_{ij})$ is
$(k,\,\ell)$-connected we have ${\rm
range}\,[(e_{k}-e_{\ell})\otimes I_{n}]\subset({\rm
null}\,L)^{\perp}={\rm range}\,L$ because $L$ is Hermitian.
Therefore we can find a matrix $R\in\Complex^{(nq)\times n}$ such
that $LR=(e_{k}-e_{\ell})\otimes I_{n}$. Then by
Lemma~\ref{lem:two} we can write
\begin{eqnarray*}
y_{ij}(t)\equiv 0\ \mbox{for all}\ (i,\,j) &\implies& L\xi(t)\equiv 0\\
&\implies& R^{*}L\xi(t)\equiv 0\\
&\implies& [(e_{k}-e_{\ell})\otimes I_{n}]^{T}\xi(t)\equiv 0\\
&\implies& x_{k}(t)\equiv x_{\ell}(t)\,.
\end{eqnarray*}
Hence the array $[(C_{ij}),\,A]$ is $(k,\,\ell)$-observable.

{\em Obs.$\implies$Con.} Suppose that $\Gamma(W_{ij})$ is not
$(k,\,\ell)$-connected. Then there must exist a vector
$\zeta\in(\Complex^{n})^{q}$ that satisfies both $L\zeta=0$ and
$[(e_{k}-e_{\ell})\otimes I_{n}]^{T}\zeta\neq 0$. Choose the
initial conditions $x_{i}(0)$ of the systems~\eqref{eqn:array} so
as to satisfy $[x_{1}(0)^{T}\ x_{2}(0)^{T}\ \cdots\
x_{q}(0)^{T}]^{T}=\zeta$. Then by Lemma~\ref{lem:two} we have
$y_{ij}(t)\equiv 0$ for all $(i,\,j)$. However,
$x_{k}(t)\not\equiv x_{\ell}(t)$ because $[(e_{k}-e_{\ell})\otimes
I_{n}]^{T}[x_{1}(0)^{T}\ x_{2}(0)^{T}\ \cdots\
x_{q}(0)^{T}]^{T}\neq 0$. I.e., the array $[(C_{ij}),\,A]$ cannot
be $(k,\,\ell)$-observable.
\end{proof}

\vspace{0.12in}

The proofs of the next two theorems are almost identical. We
therefore prove only the latter.

\begin{theorem}\label{thm:klobsV}
If the array $[(C_{ij}),\,A]$ is $(k,\,\ell)$-observable then all
the eigengraphs $\Gamma(C_{ij}V_{\sigma})$ are
$(k,\,\ell)$-connected.
\end{theorem}

\begin{theorem}\label{thm:kldet}
If the array $[(C_{ij}),\,A]$ is $(k,\,\ell)$-detectable then all
the eigengraphs $\Gamma(C_{ij}V_{\sigma})$ with ${\rm
Re}\,\mu_{\sigma}\geq 0$ are $(k,\,\ell)$-connected.
\end{theorem}

\begin{proof}
Suppose that for some $\sigma\in\{1,\,2,\,\ldots,\,m\}$ the graph
$\Gamma(C_{ij}V_{\sigma})$ is not $(k,\,\ell)$-connected and ${\rm
Re}\,\mu_{\sigma}\geq 0$. Then there exists
$\eta\in(\Complex^{n_{\sigma}})^{q}$ that satisfies
$[(e_{k}-e_{\ell})\otimes I_{n_{\sigma}}]^{T}\eta\neq 0$ and
$L_{\sigma}\eta=0$, where $L_{\sigma}={\rm
lap}\,\Gamma(C_{ij}V_{\sigma})$. Let $\zeta\in(\Complex^{n})^{q}$
be defined as $\zeta=[I_{q}\otimes V_{\sigma}]\eta$. We can write
\begin{eqnarray*}
[(e_{k}-e_{\ell})\otimes I_{n}]^{T}\zeta
&=&[(e_{k}-e_{\ell})\otimes I_{n}]^{T}[I_{q}\otimes V_{\sigma}]\eta\\
&=&V_{\sigma}[(e_{k}-e_{\ell})\otimes I_{n_{\sigma}}]^{T}\eta\\
&\neq&0
\end{eqnarray*}
because $[(e_{k}-e_{\ell})\otimes I_{n_{\sigma}}]^{T}\eta\neq 0$
and the matrix $V_{\sigma}$ is full column rank. Let $L={\rm
lap}\,\Gamma(W_{ij})$. Recall that we have
$\zeta^{*}L\zeta=\left(1+|\mu_{\sigma}|^2+\cdots+|\mu_{\sigma}|^{2(n-1)}\right)\eta^{*}L_{\sigma}\eta$
(see the proof of Lemma~\ref{lem:three}). Hence $L_{\sigma}\eta=0$
implies $L\zeta=0$ because $L$ is Hermitian positive semidefinite.
Let us choose now the initial conditions $x_{i}(0)$ of the
systems~\eqref{eqn:array} so as to satisfy $[x_{1}(0)^{T}\
x_{2}(0)^{T}\ \cdots\ x_{q}(0)^{T}]^{T}=\zeta$. Recall that each
nonzero $x_{i}(0)$ has to be an eigenvector of $A$ with eigenvalue
$\mu_{\sigma}$ because $\zeta=[I_{q}\otimes V_{\sigma}]\eta$ (see
the proof of Theorem~\ref{thm:det}). Therefore we have
$x_{i}(t)=x_{i}(0)e^{\mu_{\sigma}t}$ for all
$i\in\{1,\,2,\,\ldots,\,q\}$. Since $L\zeta=0$ Lemma~\ref{lem:two}
gives us $y_{ij}(t)\equiv 0$ for all $(i,\,j)$. However, we have
$x_{k}(0)\neq x_{\ell}(0)$ because $[(e_{k}-e_{\ell})\otimes
I_{n}]^{T}\zeta\neq 0$. Moreover, we can write
\begin{eqnarray*}
\|x_{k}(t)-x_{\ell}(t)\|=\|x_{k}(0)e^{\mu_{\sigma}t}-x_{\ell}(0)e^{\mu_{\sigma}t}\|=|e^{\mu_{\sigma}t}|\cdot\|x_{k}(0)-x_{\ell}(0)\|\not\to
0
\end{eqnarray*}
because ${\rm Re}\,\mu_{\sigma}\geq 0$. Therefore the array
$[(C_{ij}),\,A]$ cannot be $(k,\,\ell)$-detectable.
\end{proof}

\vspace{0.12in}

As mentioned earlier, an array that is not $(k,\,\ell)$-observable
may still have all its eigengraphs $(k,\,\ell)$-connected. This we
find counterintuitive. Hence a counterexample here is appropriate.

{\em A counterexample.} Consider the pair
$[(C_{ij})_{i,j=1}^{3},\,A]$ with
\begin{eqnarray*}
A=\left[\begin{array}{rrrr}
 0 &    1 &    0 &  0\\
 0 &    0 &    0 &  0\\
 0 &    0 &    0 &  1\\
 0 &    0 &    0 &  0
\end{array}\right]
\end{eqnarray*}
and
\begin{eqnarray*}
C_{12}=\left[\begin{array}{rrrr}
 0 &    0 &    1 &  0\\
 0 &    0 &    0 &  1
\end{array}\right]\,,\qquad
C_{23}=\left[\begin{array}{rrrr}
 1 &    0 &    0 &  0\\
 0 &    1 &    0 &  0\\
 0 &    0 &    0 &  1
\end{array}\right]\,,\qquad
C_{31}=\left[\begin{array}{rrrr}
 0 &    1 &    1 &  0\\
 0 &    0 &    0 &  1
\end{array}\right]\,.
\end{eqnarray*}
(Recall $C_{ij}=C_{ji}$ and $C_{ii}=0$.) Clearly, $A$ has a single
($m=1$) distinct eigenvalue $\mu_{1}=0$. The corresponding matrix
$V_{1}$ satisfying ${\rm range}\,V_{1}={\rm
null}\,[A-\mu_{1}I_{4}]$ reads
\begin{eqnarray*}
V_{1}=\left[\begin{array}{rr}
 1 &    0\\
 0 &    0\\
 0 &    1\\
 0 &    0
\end{array}\right]\,.
\end{eqnarray*}
Now consider the solutions
\begin{eqnarray*}
x_{1}(t)=\left[\begin{array}{r}
 t\\
 1\\
 0\\
 0
\end{array}\right]\,,\qquad x_{2}(t)=\left[\begin{array}{r}
 0\\
 0\\
 0\\
 0
\end{array}\right]\,,\qquad x_{3}(t)=\left[\begin{array}{r}
 0\\
 0\\
 1\\
 0
\end{array}\right]\,.
\end{eqnarray*}
It is not difficult to check that these $x_{i}(t)$ satisfy ${\dot
x}_{i}=Ax_{i}$ as well as $C_{ij}(x_{i}(t)-x_{j}(t))\equiv 0$ for
all $(i,\,j)$. Noting $x_{2}(t)\not\equiv x_{3}(t)$ we conclude
therefore that the array is not $(2,\,3)$-observable. Now let us
see what the eigengraphs say on the matter. In fact, the 2-graph
$\Gamma(C_{ij}V_{1})$ is the only eigengraph of the array. The
associated Laplacian can be computed to equal
\begin{eqnarray*}
{\rm lap}\,\Gamma(C_{ij}V_{1})=\left[\begin{array}{rrrrrr}
0 &    0 &    0 &   0 &   0 &   0\\
0 &    2 &    0 &  -1 &   0 &  -1\\
0 &    0 &    1 &   0 &  -1 &   0\\
0 &   -1 &    0 &   1 &   0 &   0\\
0 &    0 &   -1 &   0 &   1 &   0\\
0 &   -1 &    0 &   0 &   0 &   1
\end{array}\right]
\end{eqnarray*}
whose null space is spanned by the columns of the matrix $N$ given
below
\begin{eqnarray*}
N=\left[\begin{array}{rrr}
1 &    0 &    0 \\
0 &    1 &    0 \\
0 &    0 &    1 \\
0 &    1 &    0 \\
0 &    0 &    1 \\
0 &    1 &    0
\end{array}\right]\,.
\end{eqnarray*}
Observe $[(e_{2}-e_{3})\otimes I_{2}]^{T}N=0$. Hence ${\rm
range}\,[(e_{2}-e_{3})\otimes I_{2}]\subset({\rm null}\,{\rm
lap}\,\Gamma(C_{ij}V_{1}))^{\perp}$. That is, the eigengraph
$\Gamma(C_{ij}V_{1})$ {\em is} $(2,\,3)$-connected despite the
fact that the array $[(C_{ij}),\,A]$ is {\em not}
$(2,\,3)$-observable.

To better understand $(k,\,\ell)$-connectivity we now direct our
attention to circuit theory, which has been a fruitful source of
ideas for graph theory. A significant example has been reported in
\cite{klein93} where the effective resistance between two nodes of
a resistive network is shown to be a meaningful tool to measure
distance between two vertices of a graph. This work inspires us to
employ effective conductance over an $n$-graph, which we will
eventually show to be closely related to pairwise observability.
Let us however first remember the definition of effective
conductance for a 1-graph, which will be our starting point for
generalization. Let $\Gamma$ be a 1-graph with $q$ vertices and
weight function $w$. Then ${\rm lap}\,\Gamma$ equals the node
admittance matrix of a resistive network $\N$ with $q$ nodes,
where the resistor connecting the nodes $i$ and $j$ has the
conductance value $w(v_{i},\,v_{j})=:g_{ij}$ (in mhos). For the
network $\N$ the effective conductance $\gamma_{k\ell}$ between
the nodes $k$ and $\ell$ is equal to the value of the current (in
amps) leaving a 1-volt voltage source connected to the node $k$
while the node $\ell$ is grounded, see Fig.~\ref{fig:effcond}.
Regarding the case depicted in Fig.~\ref{fig:effcond}, observe
that the effective conductance $\gamma_{14}$ satisfies the
following equation
\begin{eqnarray*}
[{\rm lap}\,\Gamma]\left[\begin{array}{c}1\\
x_{2}\\ x_{3}\\ 0\\ x_{5}\\ x_{6}\end{array}\right]=\left[\begin{array}{r}1\\
0\\ 0\\ -1\\ 0\\ 0\end{array}\right]\gamma_{14}
\end{eqnarray*}
where $x_{i}\in\Real$ denote the appropriate node voltages.
Imitating this equation yields:

\begin{figure}[h]
\begin{center}
\includegraphics[scale=0.55]{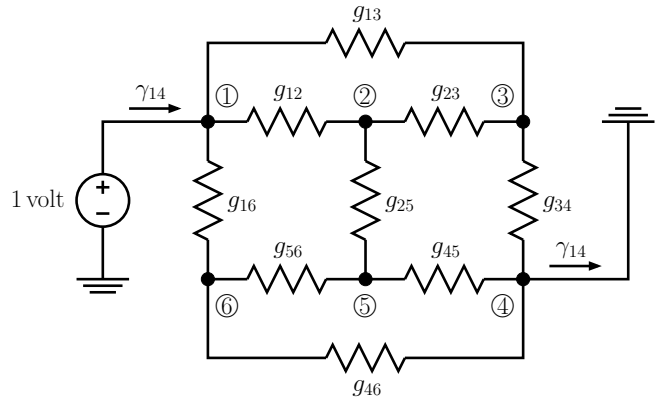}
\caption{The effective conductance between the first and fourth
nodes equals the current $\gamma_{14}$.}\label{fig:effcond}
\end{center}
\end{figure}

\begin{definition}
Given an $n$-graph $\Gamma$ with set of vertices
$\V=\{v_{1},\,v_{2},\,\ldots,\,v_{q}\}$, the {\em effective
conductance} $\Gamma_{k\ell}\in\Complex^{n\times n}$ associated to
the pair of distinct vertices $(v_{k},\,v_{\ell})$ satisfies
\begin{eqnarray}\label{eqn:effcond}
[{\rm lap}\,\Gamma]\left[\begin{array}{c}X_{1}\\ \vdots \\
X_{q}\end{array}\right]=(e_{k}-e_{\ell})\otimes
\Gamma_{k\ell}\quad\mbox{subject to}\quad X_{k}=I_{n}\ \mbox{and}\
X_{\ell}=0
\end{eqnarray}
for some $X_{i}\in\Complex^{n\times n}$.
\end{definition}

It is not evident that this definition is unambiguous. Hence we
have to make sure that effective conductance always exists,
preferably uniquely. To be able to do this we need to introduce
some notation. Let $e_{\bar k\bar\ell}\in\Complex^{q\times(q-2)}$
denote the matrix obtained from the identity matrix $I_{q}$ by
removing the $k$th and $\ell$th columns, i.e., $e_{\bar
k\bar\ell}=[e_{1}\ \cdots\ e_{k-1}\ e_{k+1}\ \cdots\ e_{\ell-1}\
e_{\ell+1}\ \cdots\ e_{q}]$. Furthermore, for $L={\rm
lap}\,\Gamma$, where $\Gamma$ is an $n$-graph with $q$ vertices we
adopt the following shortcuts.
\begin{eqnarray*}
L_{\bar k\bar\ell,\bar k\bar\ell}&=&[e_{\bar k\bar\ell}\otimes
I_{n}]^{T}L[e_{\bar k\bar\ell}\otimes I_{n}]\,,\\
L_{\bar k\bar\ell,k}&=&[e_{\bar k\bar\ell}\otimes
I_{n}]^{T}L[e_{k}\otimes I_{n}]\,,\\
L_{k,\bar k\bar\ell}&=&[e_{k}\otimes I_{n}]^{T}L[e_{\bar
k\bar\ell}\otimes I_{n}]\,,\\
L_{k,\ell}&=&[e_{k}\otimes I_{n}]^{T}L[e_{\ell}\otimes I_{n}]\,.
\end{eqnarray*}
Lastly, $(L_{\bar k\bar\ell,\bar k\bar\ell})^{+}$ indicates the
pseudo-inverse of $L_{\bar k\bar\ell,\bar k\bar\ell}$ and ${\bf
1}_{q}\in\Complex^{q}$ is the vector of all ones.

\begin{proposition}\label{prop:effcond}
Let $\Gamma$ be an $n$-graph with set of vertices
$\V=\{v_{1},\,v_{2},\,\ldots,\,v_{q}\}$ and $L={\rm lap}\,\Gamma$.
For each pair of (distinct) vertices $(v_{k},\,v_{\ell})$ the
effective conductance $\Gamma_{k\ell}$ uniquely exists and
satisfies $\Gamma_{\ell k}=\Gamma_{k\ell}=\Gamma_{k\ell}^{*}\geq
0$. In particular,
\begin{eqnarray*}
\Gamma_{k\ell}=L_{k,k}-L_{k,\bar k\bar\ell}\,(L_{ {\bar k}{\bar
\ell},{\bar k}{\bar \ell}})^{+}L_{{\bar k}{\bar \ell},k}\,.
\end{eqnarray*}
\end{proposition}

\begin{proof} Let $(v_{k},\,v_{\ell})$ be given. For $q=2$ the result follows trivially,
since $\Gamma_{12}=w(v_{1},\,v_{2})$, where $w$ is the weight
function associated to $\Gamma$. In the sequel we will consider
the case $q\geq 3$.

{\em Existence.} We first show that \eqref{eqn:effcond} can always
be solved. For $E\in\Complex^{n(q-2)\times n}$ consider the
equation
\begin{eqnarray}\label{eqn:lock}
L_{ {\bar k}{\bar \ell},{\bar k}{\bar \ell}}\,E+L_{{\bar k}{\bar
\ell},k}=0\,.
\end{eqnarray}
Note that a solution $E$ exists for \eqref{eqn:lock} if ${\rm
range}\,L_{ {\bar k}{\bar \ell},{\bar k}{\bar \ell}}\supset {\rm
range}\,L_{{\bar k}{\bar \ell},k}$, which is equivalent to ${\rm
null}\,L_{ {\bar k}{\bar \ell},{\bar k}{\bar \ell}}\subset {\rm
null}\,L_{k,{\bar k}{\bar \ell}}$ since $(L_{ {\bar k}{\bar
\ell},{\bar k}{\bar \ell}})^{*}=L_{ {\bar k}{\bar \ell},{\bar
k}{\bar \ell}}$ and $(L_{{\bar k}{\bar \ell},k})^{*}=L_{k,{\bar
k}{\bar \ell}}$. Let us now take an arbitrary
$\eta\in\Complex^{n(q-2)}$ satisfying $L_{{\bar k}{\bar
\ell},{\bar k}{\bar \ell}}\,\eta=0$ and define $\zeta=[e_{{\bar
k}{\bar \ell}}\otimes I_{n}]\eta$. We can write
\begin{eqnarray*}
\zeta^{*}L\zeta=\eta^{*}[e_{\bar k\bar\ell}\otimes
I_{n}]^{T}L[e_{\bar k\bar\ell}\otimes I_{n}]\eta=\eta^{*}L_{{\bar
k}{\bar \ell},{\bar k}{\bar \ell}}\,\eta=0
\end{eqnarray*}
which implies $L\zeta=0$ since $L$ is Hermitian positive
semidefinite. Observe
\begin{eqnarray*}
L_{k,{\bar k}{\bar \ell}}\,\eta=[e_{k}\otimes I_{n}]^{T}L[e_{\bar
k\bar\ell}\otimes I_{n}]\eta=[e_{k}\otimes I_{n}]^{T}L\zeta=0\,.
\end{eqnarray*}
Therefore ${\rm null}\,L_{ {\bar k}{\bar \ell},{\bar k}{\bar
\ell}}\subset {\rm null}\,L_{k,{\bar k}{\bar \ell}}$ and we can
indeed find $E$ satisfying \eqref{eqn:lock}.

Choose now some $E$ satisfying \eqref{eqn:lock} and define
\begin{eqnarray}\label{eqn:stock}
\Gamma_{kl}=L_{k,{\bar k}{\bar\ell}}\,E+L_{k,k}\,.
\end{eqnarray}
Observe that we have $[{\bf 1}_{q}\otimes I_{n}]^{T}L=0$ since
${\rm null}\,L\supset\setS_{n}$ and $L^{*}=L$. Also,
$e_{k}+e_{\ell}={\bf 1}_{q}-e_{{\bar k}{\bar\ell}}{\bf 1}_{q-2}$.
We can now write using \eqref{eqn:lock} and \eqref{eqn:stock}
\begin{eqnarray}\label{eqn:barrel}
L_{\ell,{\bar k}{\bar\ell}}\,E+L_{\ell,k}&=&[e_{\ell}\otimes
I_{n}]^{T}\left(L[e_{\bar k\bar\ell}\otimes
I_{n}]E+L[e_{k}\otimes I_{n}]\right)\nonumber\\
&=&[({\bf 1}_{q}-e_{{\bar k}{\bar\ell}}{\bf 1}_{q-2}-e_{k})\otimes
I_{n}]^{T}\left(L[e_{\bar k\bar\ell}\otimes I_{n}]E+L[e_{k}\otimes
I_{n}]\right)\nonumber\\
&=&\underbrace{[{\bf 1}_{q}\otimes I_{n}]^{T}L}_{0}\left([e_{\bar
k\bar\ell}\otimes I_{n}]E+[e_{k}\otimes
I_{n}]\right)\nonumber\\
&&\quad -[{\bf 1}_{q-2}\otimes
I_{n}]^{T}\left(\underbrace{[e_{{\bar k}{\bar\ell}}\otimes
I_{n}]^{T}L[e_{\bar k\bar\ell}\otimes I_{n}]}_{L_{{\bar k}{\bar
\ell},{\bar k}{\bar \ell}}}E+\underbrace{[e_{{\bar
k}{\bar\ell}}\otimes I_{n}]^{T}L[e_{k}\otimes I_{n}]}_{L_{{\bar
k}{\bar\ell},k}}\right)\nonumber\\
&&\quad -\left(\underbrace{[e_{k}\otimes I_{n}]^{T}L[e_{\bar
k\bar\ell}\otimes
I_{n}]}_{L_{k,{\bar k}{\bar\ell}}}E+\underbrace{[e_{k}\otimes I_{n}]^{T}L[e_{k}\otimes I_{n}]}_{L_{k,k}}\right)\nonumber\\
&=&-[{\bf 1}_{q-2}\otimes I_{n}]^{T}\left(\underbrace{L_{ {\bar
k}{\bar \ell},{\bar k}{\bar \ell}}\,E+L_{{\bar k}{\bar
\ell},k}}_{0}\right)-\left(\underbrace{L_{k,{\bar k}{\bar\ell}}\,E+L_{k,k}}_{\Gamma_{kl}}\right)\nonumber\\
&=&-\Gamma_{kl}\,.
\end{eqnarray}
Let us define $\Xi=[X_{1}^{T}\ X_{2}^{T}\ \cdots\ X_{q}^{T}]^{T}$
with $X_{i}\in\Complex^{n\times n}$ as $\Xi = [e_{\bar
k\bar\ell}\otimes I_{n}]E+[e_{k}\otimes I_{n}]$. We claim that
this choice $\Xi$ and $\Gamma_{kl}$ defined in \eqref{eqn:stock}
together satisfy \eqref{eqn:effcond}, i.e.,
\begin{eqnarray*}\label{eqn:smoking}
L\Xi+(e_{\ell}-e_{k})\otimes\Gamma_{kl}=0\quad\mbox{subject
to}\quad X_{k}=I_{n}\ \mbox{and}\ X_{\ell}=0\,.
\end{eqnarray*}
Note that $X_{k}=[e_{k}\otimes I_{n}]^{T}\Xi=[e_{k}\otimes
I_{n}]^{T}([e_{\bar k\bar\ell}\otimes I_{n}]E+[e_{k}\otimes
I_{n}])=I_{n}$ since $e_{k}^{T}e_{\bar k\bar\ell}=0$ and
$e_{k}^{T}e_{k}=1$. Likewise, $X_{\ell}=[e_{\ell}\otimes
I_{n}]^{T}\Xi=0$ since $e_{\ell}^{T}e_{\bar k\bar\ell}=0$ and
$e_{\ell}^{T}e_{k}=0$. To establish
$L\Xi+(e_{\ell}-e_{k})\otimes\Gamma_{kl}=0$ it suffices that we
show $[e_{i}\otimes
I_{n}]^{T}(L\Xi+(e_{\ell}-e_{k})\otimes\Gamma_{kl})=0$ for all
$i\in\{1,\,2,\,\ldots,\,q\}$. Let $i\neq k,\,\ell$. Then we can
write $[e_{i}\otimes I_{n}]^{T}=[e_{i}\otimes I_{n}]^{T}[e_{\bar
k\bar\ell}\otimes I_{n}][e_{\bar k\bar\ell}\otimes I_{n}]^{T}$.
Thence, using \eqref{eqn:lock} and $e_{\bar
k\bar\ell}^{T}(e_{\ell}-e_{k})=0$ we obtain
\begin{eqnarray*}
[e_{i}\otimes I_{n}]^{T}(L\Xi+(e_{\ell}-e_{k})\otimes\Gamma_{kl})
&=&[e_{i}\otimes I_{n}]^{T}[e_{\bar k\bar\ell}\otimes
I_{n}][e_{\bar k\bar\ell}\otimes
I_{n}]^{T}(L\Xi+(e_{\ell}-e_{k})\otimes\Gamma_{kl})\\
&=&[e_{i}\otimes I_{n}]^{T}[e_{\bar k\bar\ell}\otimes I_{n}](L_{
{\bar k}{\bar \ell},{\bar k}{\bar \ell}}\,E+L_{{\bar k}{\bar
\ell},k})\\
&=&0\,.
\end{eqnarray*}
Let $i=k$. Using \eqref{eqn:stock} and
$e_{k}^{T}(e_{\ell}-e_{k})=-1$ we can write
\begin{eqnarray*}
[e_{k}\otimes I_{n}]^{T}(L\Xi+(e_{\ell}-e_{k})\otimes\Gamma_{kl})
=L_{k,{\bar k}{\bar\ell}}\,E+L_{k,k}-\Gamma_{kl}=0\,.
\end{eqnarray*}
Finally, let $i=\ell$. Using \eqref{eqn:barrel} and
$e_{\ell}^{T}(e_{\ell}-e_{k})=1$ we reach
\begin{eqnarray*}
[e_{\ell}\otimes
I_{n}]^{T}(L\Xi+(e_{\ell}-e_{k})\otimes\Gamma_{kl}) =L_{\ell,{\bar
k}{\bar\ell}}\,E+L_{\ell,k}+\Gamma_{kl}=0\,.
\end{eqnarray*}

{\em Hermitian positive semidefiniteness.} Let $[X_{1}^{T}\
X_{2}^{T}\ \cdots\ X_{q}^{T}]^{T}=\Xi$ with
$X_{i}\in\Complex^{n\times n}$ and
$\Gamma_{kl}\in\Complex^{n\times n}$ satisfy \eqref{eqn:effcond}.
Since $X_{k}=I_{n}$ and $X_{\ell}=0$ we have
$\Xi^{*}[(e_{k}-e_{\ell})\otimes \Gamma_{kl}]=\Gamma_{kl}$. We can
write
\begin{eqnarray*}
\Gamma_{kl}
&=&\Xi^{*}[(e_{k}-e_{\ell})\otimes \Gamma_{kl}]\\
&=&\Xi^{*}L\Xi\,.
\end{eqnarray*}
Since $L$ is Hermitian positive semidefinite so is $\Xi^{*}L\Xi$.

{\em Uniqueness.} Let $[X_{1}^{T}\ X_{2}^{T}\ \cdots\
X_{q}^{T}]^{T}=\Xi$ with $X_{i}\in\Complex^{n\times n}$ and
$\Gamma_{kl}\in\Complex^{n\times n}$ satisfy \eqref{eqn:effcond}.
Suppose that $\Gamma_{kl}$ is not unique. Then we can find
$[{\tilde X}_{1}^{T}\ {\tilde X}_{2}^{T}\ \cdots\ {\tilde
X}_{q}^{T}]^{T}=\tilde\Xi$ and $\tilde\Gamma_{kl}\neq\Gamma_{kl}$
that also satisfy \eqref{eqn:effcond}. Recall $X_{k}={\tilde
X}_{k}=I_{n}$ and $X_{\ell}={\tilde X}_{\ell}=0$. Using $L^{*}=L$
and $\Gamma_{kl}^{*}=\Gamma_{kl}$ one can generate the following
contradiction
\begin{eqnarray*}
{\tilde \Gamma}_{kl} &=&\Xi^{*}[(e_{k}-e_{\ell})\otimes
{\tilde\Gamma}_{kl}] =\Xi^{*}L{\tilde\Xi}
=(L\Xi)^{*}{\tilde\Xi}\\
&=&[(e_{k}-e_{\ell})\otimes \Gamma_{kl}]^{*}{\tilde\Xi}
=({\tilde\Xi}^{*}[(e_{k}-e_{\ell})\otimes \Gamma_{kl}])^{*}
=\Gamma_{kl}^{*}\\
&=&\Gamma_{kl}\,.
\end{eqnarray*}
Thanks to uniqueness we can combine \eqref{eqn:lock} and
\eqref{eqn:stock} to write {\em the} expression for effective
conductance
\begin{eqnarray*}
\Gamma_{k\ell}=L_{k,k}-L_{k,\bar k\bar\ell}\,(L_{ {\bar k}{\bar
\ell},{\bar k}{\bar \ell}})^{+}L_{{\bar k}{\bar \ell},k}\,.
\end{eqnarray*}

{\em Reciprocity.} Let $[X_{1}^{T}\ X_{2}^{T}\ \cdots\
X_{q}^{T}]^{T}=\Xi$ with $X_{i}\in\Complex^{n\times n}$ and
$\Gamma_{kl}\in\Complex^{n\times n}$ satisfy \eqref{eqn:effcond}.
Define $\tilde X_{i}=I_{n}-X_{i}$ for $i=1,\,2,\,\ldots,\,q$. Note
that $\tilde X_{\ell}=I_{n}$ and $\tilde X_{k}=0$ because
$X_{\ell}=0$ and $X_{k}=I_{n}$. Let $\tilde \Xi=[\tilde X_{1}^{T}\
\tilde X_{2}^{T}\ \cdots\ \tilde X_{q}^{T}]^{T}$. Recall that we
have $L[{\bf 1}_{q}\otimes I_{n}]=0$. We can write
\begin{eqnarray*}
L\tilde\Xi=L([{\bf 1}_{q}\otimes
I_{n}]-\Xi)=-L\Xi=(e_{\ell}-e_{k})\otimes\Gamma_{kl}
\end{eqnarray*}
which implies $\Gamma_{\ell k}=\Gamma_{k\ell}$.
\end{proof}

\vspace{0.12in}

Effective conductance turns out to be a definite indicator of
pairwise connectivity. In particular, it allows us to improve
Theorem~\ref{thm:klobs}.

\begin{theorem}\label{thm:effcond}
The following are equivalent ($k\neq\ell$).
\begin{enumerate}
\item The array $[(C_{ij}),\,A]$ is $(k,\,\ell)$-observable. \item
The interconnection graph $\Gamma(W_{ij})$ is
$(k,\,\ell)$-connected. \item The effective conductance
$\Gamma_{k\ell}(W_{ij})$ is full rank.
\end{enumerate}
\end{theorem}

\begin{proof}
{\em 1\ $\Longleftrightarrow$\ 2.} By Theorem~\ref{thm:klobs}.

{\em 2$\implies$3.} Let us employ the shortcuts
$\Gamma_{k\ell}=\Gamma_{k\ell}(W_{ij})$ and $L={\rm
lap}\,\Gamma(W_{ij})$. Suppose that $\Gamma(W_{ij})$ is
$(k,\,\ell)$-connected. Then ${\rm
range}\,[(e_{k}-e_{\ell})\otimes I_{n}]\subset({\rm
null}\,L)^{\perp}={\rm range}\,L$ which implies that we can find
matrices $\tilde X_{1},\,\tilde X_{2},\,\ldots,\,\tilde
X_{q}\in\Complex^{n\times n}$ such that
\begin{eqnarray*}
L\left[\begin{array}{c}\tilde X_{1}\\ \vdots \\
\tilde X_{q}\end{array}\right]=(e_{k}-e_{\ell})\otimes I_{n}\,.
\end{eqnarray*}
Define $\hat X_{i}=\tilde X_{i}-\tilde X_{\ell}$ for
$i=1,\,2,\,\ldots,\,q$. Note that $\hat X_{\ell}=0$ and we can
write
\begin{eqnarray}\label{eqn:ular0}
L\left[\begin{array}{c}\hat X_{1}\\ \vdots \\
\hat X_{q}\end{array}\right]=L\left(\left[\begin{array}{c}\tilde X_{1}\\ \vdots \\
\tilde X_{q}\end{array}\right]-[{\bf 1}_{q}\otimes I_{n}]\,\tilde
X_{\ell}\right)=(e_{k}-e_{\ell})\otimes I_{n}
\end{eqnarray}
since $L[{\bf 1}_{q}\otimes I_{n}]=0$. Recalling that $L$ is
Hermitian positive semidefinite we proceed as follows
\begin{eqnarray}\label{eqn:ular1}
n={\rm rank}\,[(e_{k}-e_{\ell})\otimes I_{n}]={\rm rank}\left(L\left[\begin{array}{c}\hat X_{1}\\ \vdots \\
\hat X_{q}\end{array}\right]\right)={\rm rank}\left(\left[\begin{array}{c}\hat X_{1}\\ \vdots \\
\hat X_{q}\end{array}\right]^{*}L\left[\begin{array}{c}\hat X_{1}\\ \vdots \\
\hat X_{q}\end{array}\right]\right)\,.
\end{eqnarray}
Then we write (recalling $\hat X_{\ell}=0$)
\begin{eqnarray}\label{eqn:ular2}
\left[\begin{array}{c}\hat X_{1}\\ \vdots \\
\hat X_{q}\end{array}\right]^{*}L\left[\begin{array}{c}\hat X_{1}\\ \vdots \\
\hat X_{q}\end{array}\right]=\left[\begin{array}{c}\hat X_{1}\\ \vdots \\
\hat X_{q}\end{array}\right]^{*}[(e_{k}-e_{\ell})\otimes
I_{n}]=\hat X_{k}^{*}\,.
\end{eqnarray}
Combining \eqref{eqn:ular1} and \eqref{eqn:ular2} we deduce that
$\hat X_{k}$ is nonsingular. This allows us to define $X_{i}=\hat
X_{i}\hat X_{k}^{-1}$ for $i=1,\,2,\,\ldots,\,q$. Note that
$X_{k}=I_{n}$ and $X_{\ell}=0$. Revisiting \eqref{eqn:ular0} we
can write
\begin{eqnarray*}
L\left[\begin{array}{c}X_{1}\\ \vdots \\
X_{q}\end{array}\right]=L\left[\begin{array}{c}\hat X_{1}\\ \vdots \\
\hat X_{q}\end{array}\right]\hat
X_{k}^{-1}=[(e_{k}-e_{\ell})\otimes I_{n}]\hat
X_{k}^{-1}=(e_{k}-e_{\ell})\otimes \hat X_{k}^{-1}
\end{eqnarray*}
which implies (since $X_{k}=I_{n}$ and $X_{\ell}=0$) that
$\Gamma_{kl}=\hat X_{k}^{-1}$. Clearly, $\Gamma_{kl}$ is full
rank.

{\em 3$\implies$2.} Suppose ${\rm rank}\,\Gamma_{k\ell}=n$. Then
$\Gamma_{k\ell}^{-1}$ exists. Recall that $\Gamma_{kl}$ satisfies
\eqref{eqn:effcond} for some
$X_{1},\,X_{1},\,\ldots,\,X_{q}\in\Complex^{n\times n}$. We can
write
\begin{eqnarray*}
L\left[\begin{array}{c}X_{1}\Gamma_{k\ell}^{-1}\\ \vdots \\
X_{q}\Gamma_{k\ell}^{-1}\end{array}\right]=L\left[\begin{array}{c}X_{1}\\ \vdots \\
X_{q}\end{array}\right]\Gamma_{k\ell}^{-1}=[(e_{k}-e_{\ell})\otimes
\Gamma_{k\ell}]\,\Gamma_{k\ell}^{-1}=(e_{k}-e_{\ell})\otimes
I_{n}\,.
\end{eqnarray*}
Hence ${\rm range}\,[(e_{k}-e_{\ell})\otimes I_{n}]\subset{\rm
range}\,L=({\rm null}\,L)^{\perp}$. That is, $\Gamma(W_{ij})$ is
$(k,\,\ell)$-connected.
\end{proof}

\vspace{0.12in}

Finally, we present a result on detectability, which can be
considered as an extension of the well-known PBH test.

\begin{theorem}\label{thm:PBH}
The array $[(C_{ij}),\,A]$ is $(k,\,\ell)$-detectable
($k\neq\ell$) if and only if
\begin{eqnarray}\label{eqn:PBH}
{\rm rank}\left[\begin{array}{c}A-\lambda
I_{n}\\
\Gamma_{k\ell}(W_{ij})\end{array}\right]=n\quad\mbox{for all}\quad
{\rm Re}\,\lambda\geq 0\,.
\end{eqnarray}
\end{theorem}

\begin{proof}
Suppose that $[(C_{ij}),\,A]$ is not $(k,\,\ell)$-detectable. Then
we can find some initial conditions
$x_{1}(0),\,x_{2}(0),\,\ldots,\,x_{q}(0)$ for which the solutions
$x_{i}(t)$ of the systems~\eqref{eqn:array} yield
\begin{eqnarray}\label{eqn:and}
y_{ij}(t)\equiv 0\ \mbox{for all}\ (i,\,j)\ \mbox{and}\
\|x_{k}(t)-x_{\ell}(t)\|\not\to 0\,.
\end{eqnarray}
Earlier we have discovered (in the proof of Theorem~\ref{thm:det})
that \eqref{eqn:and} implies the existences of an eigenvalue
$\mu_{\sigma}$ in the closed left-half plane (${\rm
Re}\,\mu_{\sigma}\geq 0$) and vectors
$\vartheta_{0},\,\vartheta_{1},\,\ldots,\,\vartheta_{r}\in(\Complex^{n})^{q}$
satisfying:
\begin{itemize}
\item $([I_{q}\otimes
A]-\mu_{\sigma}I_{nq})\vartheta_{\nu}=(\nu+1)\vartheta_{\nu+1}$
for all $\nu\in\{0,\,1,\,\ldots,\,r-1\}$, \item $([I_{q}\otimes
A]-\mu_{\sigma}I_{nq})\vartheta_{r}=0$, \item
$[(e_{k}-e_{\ell})\otimes I_{n}]^{T}\vartheta_{0}\neq 0$, \item
$L\vartheta_{\nu}=0$ for all $\nu\in\{0,\,1,\,\ldots,\,r\}$,
\end{itemize}
where $L={\rm lap}\,\Gamma(W_{ij})$. Let us now fix an index
$\nu\in\{0,\,1,\,\ldots,\,r\}$ that satisfies
$[(e_{k}-e_{\ell})\otimes I_{n}]^{T}\vartheta_{\nu}\neq 0$ and
$[(e_{k}-e_{\ell})\otimes I_{n}]^{T}([I_{q}\otimes
A]-\mu_{\sigma}I_{nq})\vartheta_{\nu}=0$. Such $\nu$ should exist
because $[(e_{k}-e_{\ell})\otimes I_{n}]^{T}\vartheta_{0}\neq 0$
and $([I_{q}\otimes A]-\mu_{\sigma}I_{nq})\vartheta_{r}=0\in{\rm
null}\,[(e_{k}-e_{\ell})\otimes I_{n}]^{T}$. Let $[{\tilde
\rho}_{1}^{T}\ {\tilde \rho}_{2}^{T}\ \cdots\ {\tilde
\rho}_{q}^{T}]^{T}=\vartheta_{\nu}$ with ${\tilde
\rho}_{i}\in\Complex^{n}$. Then let
$\rho_{i}={\tilde\rho}_{i}-{\tilde\rho}_{\ell}$ for
$i=1,\,2,\,\ldots,\,q$ and $\zeta_{1}=[\rho_{1}^{T}\ \rho_{2}^{T}\
\cdots\ \rho_{q}^{T}]^{T}$. Observe that
$\rho_{k}={\tilde\rho}_{k}-{\tilde\rho}_{\ell}=[(e_{k}-e_{\ell})\otimes
I_{n}]^{T}\vartheta_{\nu}\neq 0$ and $\rho_{\ell}=0$. Moreover, we
can write
\begin{eqnarray*}
[A-\mu_{\sigma}I_{n}]\rho_{k}
&=&[A-\mu_{\sigma}I_{n}][(e_{k}-e_{\ell})\otimes I_{n}]^{T}\vartheta_{\nu}\\
&=&[(e_{k}-e_{\ell})\otimes I_{n}]^{T}([I_{q}\otimes A]-\mu_{\sigma}I_{nq})\vartheta_{\nu}\\
&=&0\,.
\end{eqnarray*}
Hence $\rho_{k}$ is an eigenvector of $A$ with eigenvalue
$\mu_{\sigma}$. Lastly, $L\zeta_{1}=L(\vartheta_{\nu}-[{\tilde
\rho}_{\ell}^{T}\ {\tilde \rho}_{\ell}^{T}\ \cdots\ {\tilde
\rho}_{\ell}^{T}]^{T})=0$ because $L\vartheta_{\nu}=0$ and
$[{\tilde \rho}_{\ell}^{T}\ {\tilde \rho}_{\ell}^{T}\ \cdots\
{\tilde \rho}_{\ell}^{T}]^{T}\in\setS_{n}\subset{\rm null}\,L$.
Let now $\Gamma_{kl}=\Gamma_{kl}(W_{ij})$ be the effective
conductance satisfying \eqref{eqn:effcond} for some
$X_{1},\,X_{2},\,\ldots,\,X_{q}\in\Complex^{n\times n}$. Also,
note that $\zeta_{1}^{*}[(e_{k}-e_{\ell})\otimes
I_{n}]=\rho_{k}^{*}-\rho_{\ell}^{*}=\rho_{k}^{*}$. By
\eqref{eqn:effcond} we can now write
\begin{eqnarray*}
\rho_{k}^{*}\Gamma_{kl}\rho_{k}&=& \zeta_{1}^{*}[(e_{k}-e_{\ell})\otimes I_{n}]\Gamma_{kl}\rho_{k}\\
&=&\zeta_{1}^{*}[(e_{k}-e_{\ell})\otimes\Gamma_{kl}]\rho_{k}\\
&=&\zeta_{1}^{*}L\left[\begin{array}{c}X_{1}\\ \vdots \\
X_{q}\end{array}\right]\rho_{k}\\
&=&0
\end{eqnarray*}
since $\zeta_{1}^{*}L=(L\zeta_{1})^{*}=0$. By
Proposition~\ref{prop:effcond} the effective conductance
$\Gamma_{kl}$ is Hermitian positive semidefinite. Therefore
$\rho_{k}^{*}\Gamma_{kl}\rho_{k}=0$ yields
$\Gamma_{kl}\rho_{k}=0$. Since $[A-\mu_{\sigma}I_{n}]\rho_{k}=0$
we can write
\begin{eqnarray*}
\left[\begin{array}{c}A-\mu_{\sigma}
I_{n}\\
\Gamma_{k\ell}\end{array}\right]\rho_{k}=0
\end{eqnarray*}
which implies that \eqref{eqn:PBH} fails.

To show the other direction this time suppose that \eqref{eqn:PBH}
fails. This implies that we can find an eigenvalue $\mu_{\sigma}$
on the closed right half-plane (${\rm Re}\,\mu_{\sigma}\geq 0$)
and an eigenvector $\rho\in\Complex^{n}$ satisfying
$\Gamma_{kl}\rho=0$ and $[A-\mu_{\sigma}I_{n}]\rho=0$. Let
$X_{1},\,X_{2},\,\ldots,\,X_{q}\in\Complex^{n\times n}$ satisfy
\eqref{eqn:effcond} and define $\zeta_{2}=[{\bar\rho}_{1}^{T}\
{\bar\rho}_{2}^{T}\ \cdots\ {\bar\rho}_{q}^{T}]^{T}$ where
${\bar\rho}_{i}=X_{i}\rho$. Note that ${\bar\rho}_{k}=\rho$ and
${\bar\rho}_{\ell}=0$ because $X_{k}=I_{n}$ and $X_{\ell}=0$. By
\eqref{eqn:effcond} we can now write
\begin{eqnarray*}
L\zeta_{2}=L\left[\begin{array}{c}X_{1}\\ \vdots \\
X_{q}\end{array}\right]\rho=[(e_{k}-e_{\ell})\otimes\Gamma_{kl}]\rho=(e_{k}-e_{\ell})\otimes(\Gamma_{kl}\rho)=0\,.
\end{eqnarray*}
Let the initial conditions $x_{i}(0)$ of the
systems~\eqref{eqn:array} satisfy $[x_{1}(0)^{T}\ x_{2}(0)^{T}\
\cdots\ x_{q}(0)^{T}]^{T}=\zeta_{2}$. Since
$x_{k}(0)={\bar\rho}_{k}=\rho$ is an eigenvector for the
eigenvalue $\mu_{\sigma}$ we have $x_{k}(t)=\rho
e^{\mu_{\sigma}t}$. Also, $x_{\ell}(0)={\bar\rho}_{\ell}=0$
implies $x_{\ell}(t)\equiv 0$. Now, since $L\zeta_{2}=0$
Lemma~\ref{lem:two} gives us $y_{ij}(t)\equiv 0$ for all
$(i,\,j)$. Moreover, we have
\begin{eqnarray*}
\|x_{k}(t)-x_{\ell}(t)\|=\|\rho
e^{\mu_{\sigma}t}-0\|=|e^{\mu_{\sigma}t}|\cdot\|\rho\|\not\to 0
\end{eqnarray*}
because ${\rm Re}\,\mu_{\sigma}\geq 0$. Therefore the array
$[(C_{ij}),\,A]$ cannot be $(k,\,\ell)$-detectable.
\end{proof}

\section{Conclusion}

In this paper we studied the observability of an array of LTI
systems with identical individual dynamics, where an array was
called observable when identically zero relative outputs implied
identical solutions for the individual systems. In our setup the
relative output for each pair of units admitted a (possibly)
different matrix. This incommensurability of the output matrices
made it necessary to study the observability of the array via the
connectivity of a matrix-weighted interconnection graph instead of
the usual scalar-weighted topologies. In the first part of the
paper we established the equivalence between the observability of
an array and the connectivity of its interconnection graph. In
addition we showed that the observability of an array could be
studied also through the connectivity of the so-called
eigengraphs, each of which corresponded to a particular eigenvalue
of the system matrix of the individual dynamics.

In the second part we investigated the pairwise observability of
an array, where an array was called $(k,\,\ell)$-observable when
identically zero relative outputs implied identical solutions for
the $k$th and $\ell$th individual systems. There too we addressed
the problem from the graph connectivity point of view. Our
findings were partially parallel to those in the first part. In
particular, we obtained the equivalence between the
$(k,\,\ell)$-observability of an array and the
$(k,\,\ell)$-connectivity of its interconnection graph. However,
in contrast with the first part, the $(k,\,\ell)$-observability of
an array was not in general guaranteed by the
$(k,\,\ell)$-connectivity of its eigengraphs. Moreover, we showed
that pairwise observability could be studied via (matrix-valued)
effective conductance, which was obtained from the interconnection
graph by treating its Laplacian as the node admittance matrix of
some resistive network, where nodes were connected by resistors
with matrix-valued conductances. We found that an array was
$(k,\,\ell)$-observable if and only if the associated effective
conductance was full rank.

\bibliographystyle{plain}
\bibliography{references}
\end{document}